\documentclass[a4paper,11pt]{article} 
\usepackage{hyperref}
\usepackage[dvipsnames]{xcolor}
\usepackage{csquotes}

\usepackage[margin=1.4 in, top=0.9 in, bottom=1.3 in]{geometry}

\usepackage{tocloft}

\usepackage[english]{babel}
\usepackage{amsfonts}
\usepackage{mathrsfs}
\usepackage{bbm}
\usepackage{latexsym}
\usepackage{math dots}
\usepackage{amssymb}
\usepackage{mathtools}
\usepackage{relsize}
\usepackage[makeroom]{cancel}

\usepackage{enumitem}
\setlist{nolistsep}
\usepackage{amsthm}
\usepackage[capitalize]{cleveref}

\newtheoremstyle{plain}{3mm}{3mm}{\slshape}{}{\bfseries}{.}{.5em}{}
\newtheoremstyle{definition}{2mm}{2mm}{}{}{\bfseries}{.}{.5em}{}
\theoremstyle{plain}
\newtheorem{Theorem}{Theorem}[section]

\newtheorem{Lemma}[Theorem]{Lemma}
\newtheorem{Proposition}[Theorem]{Proposition}
\newtheorem{Corollary}[Theorem]{Corollary}
\newtheorem{Question}[Theorem]{Question}
\theoremstyle{definition}

\newtheorem{Remark}[Theorem]{Remark}

\numberwithin{equation}{section}

\allowdisplaybreaks

\addto\captionsenglish{}

\newcommand{\N}{\mathbb{N}}

\newcommand{\R}{\mathbb{R}}

\newcommand{\Q}{\mathbb{Q}}
\newcommand{\T}{\mathbb{T}}
\newcommand{\Oh}{{\rm O}}
\newcommand{\oh}{{\rm o}}

\newcommand{\define}[1]{{\itshape #1}}

\renewcommand{\epsilon}{\varepsilon}
\renewcommand{\leq}{\leqslant}
\renewcommand{\geq}{\geqslant}

\renewcommand{\colon}{\nobreak\mskip2mu\mathpunct{}\nonscript\mkern-\thinmuskip{:}\mskip6muplus1mu\relax}

\newcommand{\targets}{\mathcal{B}}
\newcommand{\1}{1}

\renewcommand{\d}{~\mathsf{d}}

\newcommand{\rsout}[1]{{\color{red}\sout{#1}}}
\newcommand{\rcancel}[1]{{\color{red}\xcancel{#1}}}

\usepackage[normalem]{ulem}

\newcommand{\eah}{\mathcal{H}_{\mathrm{e.a.{}}}}
\newcommand{\h}{\mathcal{H}_{\mathrm{i.o.}}}
\newcommand{\G}{\Xi}
\newcommand\eq[2]{
\begin{equation}\label{eq:#1}{#2}\end{equation}}
\newcommand {\equ}[1]     {\eqref{eq:#1}}

\newcommand\ssm{\smallsetminus}

\newcommand{\ignore}[1]{}

\begin{document}

\title{\textbf{Zero--one laws for eventually always hitting points in rapidly mixing systems}}
\author{Dmitry\ Kleinbock \and Ioannis\ Konstantoulas \and Florian\ K.\ Richter}
\date{\small \today}
\maketitle
\begin{abstract}
In this work we study the set of eventually always hitting points in shrinking target systems. These are points whose long orbit segments eventually hit the corresponding shrinking targets for all future times. 
We focus our attention on systems where translates of targets exhibit near perfect mutual independence, such as Bernoulli schemes and the Gau{\ss} map.
For such systems, we present tight conditions on the shrinking rate of the targets so that the set of eventually always hitting points is a null set (or co-null set respectively).
\end{abstract}

\paragraph{Acknowledgments:}
This work grew out of the AMS Mathematics Research Communities workshop 
``Dynamical Systems:
Smooth, Symbolic, and Measurable” in June 2017; the authors are grateful to the organizers and to the AMS.  Thanks are also due to  Jon Chaika for helpful remarks during the early stages of the project, {to Tomas Persson for answering questions about recent improvements of our results in \cite{HKKP},} and to an anonymous referee for valuable comments that led to several improvements. This material is based upon work supported by the National Science Foundation under Grant Number DMS-1641020 and DMS-1926686. The first-named author was supported in part by NSF grants DMS-1600814 and DMS-1900560. The paper was finalized during the first- and third-named authors' stay at the Institute for Advanced Study (Princeton, NJ), whose support and hospitality is gratefully acknowledged.

\small
\tableofcontents
\thispagestyle{empty}
\normalsize


\section{Introduction}
\label{sec_intro}

Let $(X,\mu,T)$ be a measure preserving system, and let $\targets=\{B_n: n\in \N\}$ be a sequence of {measurable} subsets of $X$. The \define{hitting set} $\h({\targets})$ is defined as the set of $x\in X$ such that \eq{hitting}{T^nx\in B_n\text{ for infinitely many }n\in \N.}  
If $\sum_n\mu(B_n) $ is finite, it follows from the Borel--Cantelli Lemma that $\h({\targets})$ has measure zero. 
Conversely, if ${\sum_n\mu(B_n)}$ is infinite then in certain settings the hitting set $\h(\targets)$ has full measure.
Results pertaining to this dichotomy, where
\eq{hittingcriterion}{\sum_n\mu(B_n) \begin{cases} < \infty\\ = \infty\end{cases} \iff \quad\h({\targets})\text{ has }\begin{cases} \text{zero}\\ \text{full} \end{cases}\text{measure,}}
are referred to as dynamical Borel-Cantelli lemmas.

The earliest result of this type is due to Kurzweil \cite{Ku}. He proved that for $X = [0,1]$  and $T$ a rotation by $\alpha$, \equ{hittingcriterion} holds for any sequence of \define{nested} intervals ($B_1\supset B_2\supset \ldots$) if and only if $\alpha$ is badly approximable.
Later, there was an important paper of Philipp \cite{Ph} in which it is shown that \equ{hittingcriterion} holds in the cases where $X = [0,1]$, $\targets$ consists of (not necessarily nested) intervals, and $T$ is either the map $x\mapsto  \beta x \bmod 1$ or the Gau{\ss} map $x\mapsto  1/ x \bmod 1$. See e.g.\ \cite{S, KM, CK, HNPV, Kelmer17, KY} for further results, and \cite{At} for a  survey.

Let us say that   $(X,\mu, T,\targets)$  is a \define{shrinking target system} 
if the sets $B_n$ are nested\footnote{
We remark that shrinking target systems with a nested sequence of targets are sometimes also referred to as \define{monotone shrinking target systems} in the literature.} and  \eq{null}{\lim_{n\to\infty}\mu(B_n)=0.}
For $m\in\N$, write $O_m(x)\coloneqq\{Tx, T^2x,\ldots,T^{m}x\}$ for the $m$-th orbit segment of a point $x\in X$ under the transformation $T$.
Certainly, if $x$ belongs to $\h({\targets})$ then $O_m(x) \cap B_m \ne \varnothing$ for infinitely many $m$. On the other hand, if $O_m(x) \cap B_m \ne \varnothing$ for infinitely many $m$ then either  $x\in\h({\targets})$ or $T^mx\in\bigcap_{n\in\N}B_n$ 
{for some $m$}. 
Thus, under the additional assumption \equ{null}, $\h({\targets})$ coincides almost everywhere with the set
\eq{orbits}{\big\{x\in X : O_m(x) \cap B_m \ne \varnothing~\text{infinitely often}\big\}.}

In this paper, we study a natural variation of the set defined in \equ{orbits}.
Following the terminology introduced {by Kelmer}  \cite{Kelmer17}, we define the  \define{eventually always hitting set} $\eah({\targets})$ to be the set of $x\in X$ such that for all but finitely many $m\in \N$ there exists $n\in \{1,\ldots,m\}$ such that $T^{n} x\in B_m$.
Equivalently,
\begin{equation}
\label{eqn_eah}
\eah({\targets}):= \big\{x\in X : O_m(x) \cap B_m \ne \varnothing~\text{eventually always}\big\}.
\end{equation}
 
By comparing \equ{orbits} and \eqref{eqn_eah}, we see that up to a set of measure zero the eventually always hitting property is a strengthening of \equ{hitting}.
It is also not hard to show that 
in any ergodic shrinking target system, the set of eventually always hitting points obeys a zero--one law (see \cref{prop_01law} and \cref{cor_01law} below).
It is therefore natural to ask:
\begin{displayquote}
Under what conditions on the shirking rate of the size of the targets in $\targets$ can one expect $\eah(\targets)$ to have zero or full measure respectively? 
\end{displayquote}

This question has already been addressed for certain special classes of shrinking target systems.
Bugeaud and Liao   \cite{BL} looked at   maps $x\mapsto  \beta x \bmod 1$ on   $X = [0,1]$  and computed the Hausdorff dimension of   sets $\eah(\targets)$ for families of rapidly shrinking targets $\targets$.
{In the set-up of \cite{Kelmer17}, $X$ is the unit tangent bundle of a finite volume hyperbolic manifold of constant negative curvature,  $T$ is the time-one map of the geodesic flow on $X$, and $\targets$ consists of rotation-invariant subsets of $X$. Under these conditions, it was shown that $\eah(\targets)$ has full measure whenever the series $\sum_{j=1}^\infty \frac1{2^j\mu(B_{2^j})} $ diverges.
This was later generalized by Kelmer and Yu \cite{KY} to higher rank homogeneous spaces, and by Kelmer and Oh \cite{KO} to the set-up of actions on geometrically finite hyperbolic manifolds of infinite volume.
More recently, several results in this direction were obtained by Kirsebom,  Kunde and Persson \cite{KKP} for some classes of interval maps, including the doubling map, some quadratic maps,  {the Gau{\ss} map,} and the Manneville-Pomeau map. {See comments after Corollaries~\ref{cor_targetindep_2} and \ref{cor_targetindep_3} below for a comparison of some results from \cite{KKP} with our results.}

\subsection{The main technical result}

Our main technical result concerns systems whose targets satisfy a long-term independence property that arises in connection with rapid mixing. In such cases, we give sufficient conditions for the set of eventually always hitting points to either have zero or full measure.
The class of systems to which this applies contains several relevant examples, such as product systems, Bernoulli schemes and the Gau{\ss} map.

The long-term independence property that we impose in our theorem asserts{, roughly speaking,} that any target $B_m\in\targets$ becomes ``evenly spread out'' under the transformation $T$ in the sense that
$\mu(B_n\cap T^{-k}B_m)\approx \mu(B_n)\mu(B_m)$ {for all $k\geq k(n,m)$}, where $k(n,m)$ depends on $n$ and $m$. The precise formulation is 
more technical and involves 
\begin{equation}
\label{eqn_thetanm}
\begin{aligned}
\G_{n,m}:= \text{ the algebra of subsets of $X$\qquad}\\\text{ generated by }\{T^{-j}B_i: 1\leq i\leq m,\ 1\leq j\leq n \}.\end{aligned}
\end{equation}
It states the following:
\begin{equation}
\label{eqn_targetindependence_longterm_almost}
\begin{aligned}
{\text{For all $m,n\in\N$ with $n\leq m$, all $A\in\G_{n,m}$, and all $B\in\G_{m,m}$ one has}}\\
\big|\mu\big(A\cap T^{-(n+{F}(m))} B\big)- \mu(A)\mu(B)\big|\leq \eta(m)\mu(A)\mu(B),\qquad
\end{aligned}
\end{equation}
where $\eta\colon\N\to [0,1]$ is some function satisfying $\lim_{m\to\infty}\eta(m)=0$, and ${F}\colon\N\to \N$
 is another function satisfying
\begin{equation}
\label{eqn_targetindependence_longterm_almost_mathcalH}
{F}(m)\, \leq\, \frac{1}{(\log m)^{1+\delta}\mu(B_m)}\text{ for some $\delta>0$ and all large enough $m\in\N$}.
\end{equation}
We also define the set
\begin{equation}
\label{eqn_E_m}
E_m\coloneqq \{x\in X: O_m(x)\cap B_m= \varnothing\},
\end{equation}
which describes the collection of all points in $X$ for which none of the first $m$ iterates under the transformation $T$ visits the target $B_m$.
Note that
\eq{limsup}{
X\ssm \eah(\targets)
= \limsup E_n \coloneqq \mathsmaller{\bigcap_{n\in \N}\bigcup_{m\geq n}} E_m.
}
	
\begin{Theorem}
\label{thm_targetindep_longterm_almost}
Let $(X,\mu,T,\targets)$ be a shrinking target system satisfying \eqref{eqn_targetindependence_longterm_almost}. If
$$
\sum_{n=1}^\infty \frac{\mu(E_{n})^{1-\epsilon}}{n}<\infty
$$
for some $\epsilon>0$, then $\eah(\targets)$ has full measure. On the other hand, if
$$
\sum_{n=1}^\infty\frac{\mu(E_{n})}{n}=\infty,
$$
then $\eah(\targets)$ has zero measure.
\end{Theorem}

\subsection{Product systems}\label{prod}

For our first application of \cref{thm_targetindep_longterm_almost}, fix an arbitrary probability space $(Y,\nu)$, and let $A_1\supset A_2\supset\ldots$ be a sequence of measurable subsets of $Y$ with $\nu(A_n)\to 0$ as $n\to\infty$. Consider the shrinking target system $(X
,\mu
,T
,\targets
)$, where $X
 \coloneqq Y^{\N\cup\{0\}}$, $\mu
\coloneqq \nu^{\otimes \N\cup\{0\}}$, $T
\colon X
\to X
$ denotes the left shift, and the shrinking targets $\targets
 \coloneqq \{B_1\supset B_2\supset\ldots\}$ are defined as $B_n \coloneqq \{x\in X: x[0]\in A_n\}$.
The elements in $\targets$ have the convenient property that
\begin{equation}
\label{eqn_targetindependence}
\mu(B_n\cap T^{-k}  B_m)= \mu(B_n)\mu(B_m),\quad\forall \,k,n,m \in\N,
\end{equation}
which immediately implies that the shrinking target system $(X
,\mu
,T
,\targets
)$ satisfies condition \eqref{eqn_targetindependence_longterm_almost} with $\eta(m)=0$ and ${F}(m)=0$ for all $m\in\N$.

\begin{Theorem}
\label{thm_targetindep_1}
Let $(X,\mu,T,\targets)$ be the shrinking target system described above. If
$$
\sum_{n=1}^\infty \frac{\big(1-\mu(B_{n})\big)^{n(1-\epsilon)}}{n}<\infty
$$
for some $\epsilon>0$, then $\eah(\targets)$ has full measure.
On the other hand, if
$$
\sum_{n=1}^\infty\frac{\big(1-\mu(B_{n})\big)^n}{n}=\infty,
$$
then $\eah(\targets)$ has zero measure.
\end{Theorem}

From \cref{thm_targetindep_1} one can derive the following corollary.

\begin{Corollary}
\label{cor_targetindep_1}
Let $(X
,\mu
,T
,\targets
)$ be {as in Theorem \ref{thm_targetindep_1}.
Suppose that} there exists $C>1$ such that for all but finitely many $m$ one has
$$
\mu(B_m)\geq \frac{C \log\log m}{m};
$$
then $\eah(\targets)$ has full measure. If, on the other hand,
$$
\mu(B_m)\leq \frac{\log\log m}{m}
$$
for all but finitely many $m$, then $\eah(\targets)$ has zero measure.
\end{Corollary}



\subsection{Bernoulli schemes}
\label{sec_bernoulli}

Another class of systems that satisfy \eqref{eqn_targetindependence_longterm_almost} for a natural choice of shrinking targets are Bernoulli schemes. Let $(X,T)$ denote the full symbolic shift in $2$ letters\footnote{The same results hold for shifts on $ \{0,\dots,b-1\}^{\N\cup\{0\}}$ for any integer $b > 2$; we chose to restrict ourselves to the case $b=2$ to simplify the presentation.}, that is, $X\coloneqq \{0,1\}^{\N\cup\{0\}}$. Let $T\colon X\to X$ be the left shift {on $X$, and} denote by $\mu$ the $(1/2,1/2)$-Bernoulli measure on $X$. {Given} a non-decreasing unbounded sequence of indices $(r_m)_{m\in\N}$, consider the corresponding sequence of shrinking targets $\targets=\{B_1\supset B_2\supset\cdots\}$ defined as
\eq{intervals}{
B_m\coloneqq\big\{x\in X
: x[0]=x[1]=\cdots=x[r_m-1]=0\big\}, \qquad\forall\, m\in\N.
}
Note that $\mu(B_m) = 2^{-r_m}$. It is then straightforward to verify that the resulting shrinking target system $(X,\mu,T,\targets)$ satisfies condition \eqref{eqn_targetindependence_longterm_almost} with $\eta(m)=0$ and ${F}(m)=r_m$ for all $m\in\N$.

\begin{Theorem}
\label{thm_bernoulli}
Let $(X,\mu,T,\targets)$ be as above, and assume that
either one of the following two conditions is satisfied:
\begin{align}
&\exists\,D>2\text{ such that }\mu(B_m)\geq \tfrac{D\log\log m}{m}\text{ for all but finitely many }m\in\N; \label{loglog}
\\
&\text{$\exists\,\tau>1$ such that }\mu(B_m)\leq \tfrac{1}{(\log m)^{\tau}}\text{ for all but finitely many }m\in\N.\label{tau}
\end{align}
If
\begin{equation}
\label{eqn_bern_1}
\sum_{n=1}^\infty \frac{\big(1-\mu(B_{n})\big)^{\frac{n(1-\epsilon)}{2}}}{n}<\infty
\end{equation}
for some $\epsilon>0$, then $\eah(\targets)$ has full measure.
On the other hand, if
\begin{equation}
\label{eqn_bern_2}
\sum_{n=1}^\infty\frac{\big(1-\mu(B_{n})\big)^{\frac{n}{2}}}{n}=\infty,
\end{equation}
then $\eah(\targets)$ has zero measure.
\end{Theorem}


{Note that \eqref{loglog} implies $\sum_{n=1}^\infty\frac1n\big(1-\mu(B_{n})\big)^{\frac{n(1-\epsilon)}{2}}<\infty$, which means that if we are in case \eqref{loglog}, then automatically $\eah(\targets)$ has full measure.
This observation is part of a dichotomy that is described by the following analogue of \cref{cor_targetindep_1}.
}

\begin{Corollary}
\label{cor_targetindep_2}
Let $(X
,\mu
,T
,\targets
)$ be as in \cref{thm_bernoulli}. If 
\eqref{loglog} holds,
then $\eah(\targets)$ has full measure. If, on the other hand,
\eq{2loglog}{
\mu(B_m)\leq \frac{2\log\log m}{m}\text{ 
for all but finitely many }m\in\N,} then $\eah(\targets)$ has zero measure.
\end{Corollary}

{We remark  that it is proved in \cite[Theorem 1]{KKP} that if $(X
,\mu
,T
)$ is as in \cref{thm_bernoulli} and $\targets = \{B_m\}$ is a family of nested intervals, then 
 $$
\mu\big(\eah(\targets)\big) = \begin{cases}\ 1  &\text{if }\mu(B_m) \ge \frac{c(\log m)^2}m \text{ 
for some $c > 0$ and large enough }m;\\
\ 0  &\text{if }\mu(B_m)  \le \,\,
c/
m\quad\,  \text{ for some $c > 0$ and large enough }m.
\end{cases}
$$
Note that the above inequalities are significantly stronger than \eqref{loglog}  and \equ{2loglog} respectively; on the other hand, our method is applicable only to 
specific families of targets given by \equ{intervals}.} 

\subsection{The Gau{\ss} map}
\label{sec_cft}

The 
\define{Gau{\ss} map} is the map $T$ on the interval $X\coloneqq [0,1]$ defined as $$ T(x)\coloneqq\begin{cases}\frac{1}{x}- \left\lfloor \frac{1}{x}\right\rfloor& \text{if}~x\neq 0,\\0&\text{if}~x=0.\end{cases}$$
There is an explicit $T$-invariant Borel probability measure on $[0,1]$ called the \define{Gau{\ss} measure} (cf.\ \cite[Lemma 3.5]{EW11}):
$$
\mu(B)\coloneqq \frac{1}{\log 2}\int_B \frac{\d x}{1+x},\qquad\text{for all measurable $B\subset[0,1]$.}
$$ 
The Gau{\ss} map and the Gau{\ss} measure are tightly connected to the theory of continued fractions. Any irrational number $x\in[0,1]$ has a unique \define{simple continued fraction expansion}
$$
x~=~\cfrac{1}{a_1+\cfrac{1}{a_2+\cfrac{1}{a_3+_{\ddots}}}},\qquad a_1,a_2,\ldots\in\N,
$$
which we write as $[a_1,a_2,\ldots]$.
Note that if $x=[a_1,a_2,\ldots]$, then $T(x)=[a_2,a_3,\ldots]$. Thus $T$ acts as the left shift on the continued fraction representation of a number. This identification leads us to a natural shrinking target problem where the targets are determined by digit restrictions in the continued fraction expansion.
Let $(k_m)_{m\in\N}$ be a non-decreasing sequence of natural numbers, and consider the sequence of shrinking targets $\targets=\{B_1\supset B_2\supset\ldots\}$ given by
\begin{equation}
\label{eqn_def_Bm_gauss}
B_m\coloneqq\{[a_1,a_2,\ldots]: a_1\geq k_m\} = [0,1/{k_m}]
\end{equation}
for all $m\in\N$;\footnote{Technically speaking, the equality in \eqref{eqn_def_Bm_gauss} is incorrect as written and should instead be $\{[a_1,a_2,\ldots]: a_1\geq k_m\} = [0,1/{k_m}]\ssm \Q$, because the set in the left-hand side consists of infinite continued fraction expansions which only yield irrational numbers. However, since all shrinking target problems that we consider are insensitive to adding or removing a zero-measure set from the targets or the underlying space, we brush aside this issue for the sake of cleaner notation and simpler expressions.} note that 
\eq{gauss}{\mu(B_{m}) = \frac{\log(1 + 1/{k_m}) }{\log 2}.}
We show in \cref{sec_gauss_proof} that the shrinking target system $(X,\mu,T,\targets)$ satisfies condition \eqref{eqn_targetindependence_longterm_almost} for any ${F}(m)$ that satisfies \eqref{eqn_targetindependence_longterm_almost_mathcalH} and $\eta(m)=\Oh\left(\exp\big(-C{\sqrt{{F}(m)}}\,\big)\right)$ for some universal constant $C>0$. {Combining this with \cref{thm_targetindep_longterm_almost} allows us to derive the following result.}

\begin{Theorem}\label{thm_gauss} Let $(X,\mu,T,\targets)$ be as described above, and assume that either there exists $\sigma<1$ such that $k_m\leq \frac{\sigma m}{\log\log m}$ for all but finitely many $m\in\N$, or there exists $\tau>0$ such that $k_m\geq (\log m)^{\tau}$ for all but finitely many $m\in\N$. If
$$
\sum_{n=1}^\infty \frac{\big(1-\mu(B_{n})\big)^{(\log 2)(1-\epsilon)n}}{n}<\infty
$$
for some $\epsilon>0$, then $\eah(\targets)$ has full measure.
On the other hand, if
$$
\sum_{n=1}^\infty\frac{\big(1-\mu(B_{n})\big)^{
{2(\log2)(1+\epsilon)n}}}{n}=\infty,
$$
for some $\epsilon>0$, then $\eah(\targets)$ has zero measure. 
\end{Theorem}

\begin{Corollary}
\label{cor_targetindep_3}
Let $(X,\mu,T,\targets)$ be as  described above. If there exists {{$C_1>1$}
} such that for all but finitely many $m$ one has
$$
{\log(1 + 1/{k_m})}\geq \frac{C_1\log\log m}{m},
$$
then $\eah(\targets)$ has full measure. On the other hand, if there exists  
{$C_2<1/2$} such that
$$
{\log(1 + 1/{k_m})}\leq \frac{C_2 \log\log m}{m}
$$
for all but finitely many $m$, then $\eah(\targets)$ has zero measure.
\end{Corollary}

{We remark  that  \cite[Theorem 3]{KKP}, which asserts that 
$$
\mu\big(\eah(\targets)\big) = \begin{cases}\ 1  &\text{if }k_m =  \frac{cm}{(\log m)^2} \text{ where $c > 0$ is sufficiently small},\\
\ 0  &\text{if }k_m =   {cm} \qquad \text{for any } c > 0,
\end{cases}
$$
is a 
{consequence} of the above corollary.} 

\begin{Remark}
{During the period of revision for this paper, Theorems \ref{thm_bernoulli} and \ref{thm_gauss} have already been improved upon by other authors. In the recent preprint \cite{HKKP}  it is shown that if $(X,\mu,T)$ is as in \cref{thm_bernoulli} and the targets $\targets=\{B_n : n\in\N\}$ are as in \equ{intervals} with the additional assumption that $n\mapsto n\mu(B_n)$ is eventually non-decreasing, then 
\begin{equation}
\label{eqn_nzol}
\sum_n \mu(B_n) e^{-\frac{n}{2} \mu(B_n) }\begin{cases} < \infty\\ = \infty\end{cases} \iff \quad\eah({\targets})\text{ has }\begin{cases} \text{full}\\ \text{zero} \end{cases}\text{measure.}
\end{equation}
The comparison between the convergence conditions in \eqref{eqn_nzol} and \cref{thm_bernoulli} provides an interesting insight. To see that the conditions given in \eqref{eqn_nzol} conform to and, in fact, improve upon the conditions in \cref{thm_bernoulli}, we must show that
\begin{equation}
\label{eqn_czol_1}
\sum_{n} \frac{\big(1-\mu(B_n)\big)^{\frac{n}{2}}}{n} = \infty \quad\implies\quad \sum_n \mu(B_n) e^{-\frac{n}{2} \mu(B_n) }= \infty,
\end{equation}
as well as
\begin{equation}
\label{eqn_czol_2}
\sum_{n} \frac{\big(1-\mu(B_n)\big)^{\frac{n (1-\epsilon)}{2}}}{n} < \infty \quad\implies\quad \sum_n \mu(B_n) e^{-\frac{n}{2} \mu(B_n) }< \infty.
\end{equation} 
We will do so under the simplifying assumption that $\frac{c}{n}\leq \mu(B_n)$ for some $c\in(0,1)$ {and all large enough $n$}. This assumption is not overly restrictive, because if, {on the contrary},  $\mu(B_n)< \frac{c}{n}$ for some $c\in (0,1)$ {and infinitely many $n$}, then it is well known that we are in the $\mu\big(\eah({\targets})\big)=0$ case (see \cite[Proposition~12]{Kelmer17}).}

{To prove \eqref{eqn_czol_1}, note that from the basic inequality $(1-x)\leq e^{-x}$, which holds for all $x\in\R$, we immediately get ${\big(1-\mu(B_n)\big)^{\frac{n}{2}}}\leq  e^{-\frac{n}{2} \mu(B_n) }$. In view of $\frac{c}{n}\leq \mu(B_n)$, this implies $\frac1n{\big(1-\mu(B_n)\big)^{\frac{n}{2}}}\leq  c^{-1} \mu(B_n)e^{-\frac{n}{2} \mu(B_n) }$ and the claim follows.
}

{
For the proof of \eqref{eqn_czol_2}, we utilize the comparison test to show that convergence of the left series implies convergence of the right series. It suffices to show
\[
\limsup_{n\to\infty}
\frac{~
\mu(B_n)e^{-\frac{n}{2} \mu(B_n)}~}{ \frac1{n}{\big(1-\mu(B_n)\big)^{\frac{n (1-\epsilon)}{2}}}} \leq 1.
\]
Since $n^{{2}/{n}}\to 1$ and $\mu(B_n)\leq 1$, the above is implied by
\[
\limsup_{n\to\infty}
\frac{~e^{-\mu(B_n)}~}{ {\big(1-\mu(B_n)\big)^{1-\epsilon}}} \leq 1.
\]
This follows from the inequality $e^{-x}\leq (1-x)^{1-\epsilon}$, which holds for all sufficiently small positive $x$.
}

{ 
Results similar to \eqref{eqn_nzol} {are also established in \cite{HKKP}} for more general interval maps $T$ of $X = [0,1]$ with a Gibbs measure $\mu$ on $X$ and some additional regularity conditions (see~\cite[Theorems~3.2~and~4.1]{HKKP}). 
}
{In parallel, results about the Gau{\ss} map were also obtained (see \cite[paragraph after Corollary 2.5]{HKKP}): If $X=[0,1]$, $T$ is the Gau{\ss} map, $\mu$ the Gau{\ss} measure, the targets $\targets$ are as in \eqref{eqn_def_Bm_gauss}, and $n\mapsto n\mu(B_n)$ is eventually non-decreasing, then 
\begin{equation*}
\sum_n \mu(B_n) e^{-n \mu(B_n) }\begin{cases} < \infty\\ = \infty\end{cases} \iff \quad\eah({\targets})\text{ has }\begin{cases} \text{full}\\ \text{zero} \end{cases}\text{measure.}
\end{equation*} 
}
{It seems plausible that the methods of  \cite{HKKP} are applicable to the set-up of \S\ref{prod}, and thus Theorem \ref{thm_targetindep_1} can also be upgraded to a necessary and sufficient condition.}
\end{Remark}

\section{General properties of $\eah$ sets}
\label{sec_general_stuff}

Before embarking on the proofs of Theorems  \ref{thm_targetindep_longterm_almost}, \ref{thm_targetindep_1}, \ref{thm_bernoulli}, and \ref{thm_gauss}, we gather in this section some general results regarding  $\eah$ sets that apply to all ergodic shrinking target systems.
In Subsection \ref{sec_01law} below, we show that all $\eah$ sets obey a zero--one law. Thereafter, in Subsections \ref{sec_gen_suff_cond} and \ref{sec_gen_nec_cond} we present general necessary and sufficient conditions for $\eah$ sets to have full measure.

\subsection{The zero--one law for eventually always hitting points}
\label{sec_01law}

We begin with showing that $\eah$ sets are essentially invariant{, a result that was obtained independently in \cite{KKP}}.

\begin{Proposition}\label{prop_01law}
Let $(X,\mu,T,\targets)$ be a (not necessarily ergodic) shrinking target system.
Then $\eah(\targets)$ is essentially invariant under $T$, that is, $$\mu\big(\eah(\targets)\triangle T^{-1}\eah(\targets)\big)=0.$$
\end{Proposition}


\begin{proof}
Let $Y\coloneqq\{x\in X: Tx\in B_n\text{ for infinitely many }n\}$. Since $B_n$ are nested, we have that $Y=\bigcap_{n\in\N} T^{-1}B_n$ and, using $\mu(B_n)\to 0$ as $n\to\infty$, it follows from the monotone convergence theorem that $Y$ has zero measure.

We now claim that if $x\in \eah(\targets)\ssm Y$, then $Tx\in \eah(\targets)$. To verify this claim recall that
$$
x\in  \eah(\targets) \quad\iff\quad O_n(x)\cap B_n\neq\varnothing\text{ for all but finitely many }n.
$$
Note also that $O_n(x)=\{Tx\}\cup O_{n-1}(Tx)$. Therefore, if $x\notin Y$, then $O_n(x)\cap B_n$ is non-empty  for cofinitely many $n$ if and only if $O_{n-1}(Tx)\cap B_n\neq\varnothing$ for cofinitely many $n$. Hence
\begin{eqnarray*}
x\in  \eah(\targets)\ssm Y &\implies & O_n(x)\cap B_n\neq\varnothing\text{ for all but finitely many }n
\\
&\implies& O_{n-1}(Tx)\cap B_n\neq\varnothing\text{ for all but finitely many }n
\\
&\implies & O_{n-1}(Tx)\cap B_{n-1}\neq\varnothing\text{ for all but finitely many }n
\\
&\implies & O_{n}(Tx)\cap B_{n}\neq\varnothing\text{ for all but finitely many }n,
\end{eqnarray*}
where in the second to last implication we have used that $B_n\subset B_{n-1}$. This proves that if $x\in \eah(\targets)\ssm Y$ then $Tx\in \eah(\targets)$. Therefore
$
\eah(\targets)\ssm Y\subset T^{-1}\eah(\targets).
$
Since $\mu(Y)=0$ and $T$ is measure preserving, we conclude that
\begin{eqnarray*}
\mu\Big(\eah(\targets)\triangle T^{-1}\eah(\targets)\Big)
&=&0.
\end{eqnarray*}
This finishes the proof.
\end{proof}

In the presence of ergodicity, all essentially invariant sets are trivial. Therefore \cref{prop_01law} implies the following corollary.
\begin{Corollary}
\label{cor_01law}
If $(X,\mu,T,\targets)$ is an ergodic shrinking target system, then $\eah(\targets)$ is either a null set or a co-null set.
\end{Corollary}

\subsection{General sufficient condition for $\mu\big(\eah(\targets)\big)=1$}
\label{sec_gen_suff_cond}


For $m,n\in\N$ define
\begin{equation}
\label{eqn_non-hitting_sets}
E_{n,m} \coloneqq \{x: O_n(x) \cap B_m = \varnothing\},
\end{equation}
which can also be written as 
\begin{equation}
\label{eqn_non-hitting_sets_2}
E_{n,m} =\bigcap_{i=1}^{n}T^{-i} B_m^c.
\end{equation}
{Note that $E_{n,m}\in \G_{n,m}$ for all $m,n\in\N$, and the sets $E_m$ defined in \eqref{eqn_E_m} coincide with $E_{m,m}$.}  
The following result is taken from \cite{Kelmer17} and plays an important role in our proof of \cref{thm_targetindep_longterm_almost}:

\begin{Lemma}[{\cite[Lemma 13]{Kelmer17}}]\label{lem_conull}
Suppose there exists a non-decreasing sequence $m_j$ such that $\sum_{j=0}^\infty \mu(E_{m_{j-1},m_j})< \infty$.
Then $\mu\big(\eah(\targets)\big)=1$.
\end{Lemma}

\subsection{General necessary condition for $\mu\big(\eah(\targets)\big)=1$}
\label{sec_gen_nec_cond}

The next result establishes a necessary condition for $\eah$ sets to have full measure, conditional under the assumption that the sets $E_{{m}}$ are asymptotically independent. 

\begin{Theorem}
\label{thm_targetindep_0}
Let $(m_j)_{j\in\N}$ be a non-decreasing sequence and $(X,\mu,T,\targets)$ a shrinking target system with the property that
\begin{equation}
\label{eqn_asymptotic_indep}
\mu(E_{m_s}\cap E_{m_t})\leq \big(1+\oh_{t\to\infty}(1)\big) \mu(E_{m_s})\mu(E_{m_t})^{1-2^{s-t+2}} + \Oh\big(\mu(E_{m_s}) v_t\big)
\end{equation}
where $(v_t)_{t\in\N}$ is a sequence of non-negative numbers satisfying $\sum_{t\in\N} v_t<\infty$. 
If $\mu\big(\eah(\targets)\big )=1$, then necessarily $\sum_{j=1}^\infty \mu(E_{m_j})<\infty$.
\end{Theorem}

For the proof of \cref{thm_targetindep_0} we need the following lemma.

\begin{Lemma}\label{lem_decay}
Let $0<q_i<1$ for $i=1,\dots, N$.  Then $$\sum_{m>n}^N \left(q_nq_m^{1-2^{n-m+2}}-q_nq_m \right)=  \Oh\left(\sum_{n=1}^N q_n\right).$$
\end{Lemma}

\begin{proof}
Recall Bernoulli's inequality, which asserts that $(1+y)^r - 1 \leq ry$ for all $r\in (0,1)$ and $y>-1$.
If we apply this inequality with $y = \frac{1}{q_m}-1$ and $r = 2^{n-m+2}$ we obtain
$$
\left(\tfrac{1}{q_m}\right)^{2^{n-m+2}}-1~\leq~  2^{n-m+2}\left(\tfrac{1}{q_m}-1\right)~\leq~ \frac{1}{2^{m-n-2}q_m}.
$$ 
This gives
\begin{eqnarray*}
\sum_{m>n}^N \left(q_nq_m^{1-2^{n-m+2}}-q_nq_m \right)
&=&
\sum_{m>n}^N q_nq_m\left(\left(\tfrac{1}{q_m}\right)^{2^{n-m+2}}-1 \right)
\\
&\leq&\sum_{m>n}^N \frac{q_n}{2^{m-n+2}}
=
\Oh\left(\sum_{n=1}^N q_n\right).
\end{eqnarray*}
\end{proof}

\begin{proof}[Proof of \cref{thm_targetindep_0}]
By way of contradiction, assume that $\sum_{j=1}^\infty \mu(E_{m_j})=\infty$.
Let $\1_j=\1_{E_{m_j}}$ denote the indicator function of $E_{m_j}$, and define $q_j \coloneqq \mu(E_{m_j})$. Consider the normalized deviation
$$
D_N = \frac{\sum_{j=1}^N \1_j}{\sum_{j=1}^N q_j} - 1.
$$
Its $L^2$-norm is
\begin{eqnarray*}
\|D_N\|_2^2
&=&\frac{2\sum_{t>s}^N \left\langle 1_s-q_s, 1_t-q_t \right\rangle}{\left(\sum_{j=1}^N q_j\right)^2}+\frac{\sum_{j=1}^N \left\langle 1_j-q_j, 1_j-q_j \right\rangle}{\left(\sum_{j=1}^Nq_j\right)^2}
\\
&=&\frac{2\sum_{t>s}^N \left\langle 1_s-q_s, 1_t-q_t \right\rangle}{\left(\sum_{j=1}^Nq_j\right)^2}+\oh_{N\to\infty}(1)
\\
&=&\frac{2\sum_{t>s}^N \left(\mu(E_{m_s}\cap E_{m_t})-q_sq_t\right)}{\left(\sum_{j=1}^Nq_j\right)^2}+\oh_{N\to\infty}(1).
\end{eqnarray*}
Fix $\epsilon>0$. As guaranteed by \eqref{eqn_asymptotic_indep}, there exists $M\in\N$ such that for all $s,t\in \N$ with $t\geq M$ one has
$$
\mu(E_{m_s}\cap E_{m_t})\leq(1+\epsilon)q_sq_t^{1-2^{s-t+2}} +  \Oh(q_s v_t).
$$
Hence
\begin{eqnarray*}
\|D_N\|_2^2
&\leq &
\tfrac{2\sum_{M<s<t<N} \left(q_sq_t^{1-2^{s-t+2}}(1+\epsilon)-q_sq_t+\Oh(q_s v_t)\right)}{\left(\sum_{j=1}^Nq_j\right)^2}+\oh_{N\to\infty}(1)
\\
&= &
\tfrac{2(1+\epsilon)\sum_{M<s<t<N} \left(q_sq_t^{1-2^{s-t+2}}-q_sq_t\right)}{\left(\sum_{j=1}^N q_j\right)^2}+
\\
&&\qquad+~\tfrac{2\epsilon\sum_{M<s<t<N} q_sq_t}{\left(\sum_{j=1}^N q_j\right)^2}+\Oh\left(\tfrac{\sum_{M<s<t<N} q_s v_t}{\left(\sum_{j=1}^N q_j\right)^2}\right)+\oh_{N\to\infty}(1)
\\
&\leq &
\tfrac{2(1+\epsilon)\sum_{M<s<t<N} \left(q_sq_t^{1-2^{s-t+2}}-q_sq_t\right)}{\left(\sum_{j=1}^Nq_j\right)^2}+\epsilon+\Oh\left(\tfrac{\sum_{t=1}^N v_t}{\sum_{j=1}^N q_j}\right)+\oh_{N\to\infty}(1).
\end{eqnarray*}
Since by assumption $\sum_{j=1}^\infty q_j = \infty$ and $\sum_{t=1}^\infty v_t<\infty$, the term $\Oh\left(\tfrac{\sum_{t=1}^N v_t}{\sum_{j=1}^N q_j}\right)$ goes to $0$ as $N\to\infty$. Also, using \cref{lem_decay}, we can control the term
$$
\tfrac{2(1+\epsilon)\sum_{M<s<t<N} \left(q_sq_t^{1-2^{s-t+2}}-q_sq_t\right)}{\left(\sum_{j=1}^Nq_j\right)^2}.
$$
Indeed, 
\begin{eqnarray*}
\tfrac{2(1+\epsilon)\sum_{M<s<t<N} \left(q_sq_t^{1-2^{s-t+2}}-q_sq_t\right)}{\left(\sum_{j=1}^Nq_j\right)^2} &\leq &
\tfrac{2(1+\epsilon)\sum_{1<s<t<N} \left(q_sq_t^{1-2^{s-t+2}}-q_sq_t\right)}{\left(\sum_{j=1}^Nq_j\right)^2}
\\
&= &
\Oh\left(\tfrac{2(1+\epsilon)\left(\sum_{j=1}^Nq_j\right)}{\left(\sum_{j=1}^Nq_j\right)^2}\right)
\\
&= &
\Oh\left(\tfrac{1}{\sum_{j=1}^Nq_j}\right)
= 
\oh_{N\to\infty}(1).
\end{eqnarray*}
This proves that $\|D_N\|_2^2\leq \epsilon+ \oh_{N\to\infty}(1)$. Since $\epsilon$ was chosen arbitrarily, we obtain $\|D_N\|_2^2= \oh_{N\to\infty}(1)$. The decay of the $L^2$-norm of $D_N$ implies that $\limsup E_{m_j}$ has full measure. Therefore $\mu(\limsup E_n)=1$,
which, in view of \equ{limsup}, 
contradicts $\mu\big(\eah(\targets)\big)=1$.
\end{proof}

\section{Proof of the main technical result}

This section is dedicated to the proof of \cref{thm_targetindep_longterm_almost} and is divided into three subsections.
In Subsection \ref{sec_meas_of_E_nm} we study the asymptotic behavior of $\mu(E_{m_{j-1},m_j})$ for certain lacunary sequences $(m_j)$, which is needed for the proof of \cref{thm_targetindep_longterm_almost} in combination with \cref{lem_conull}.
In Subsection \ref{sec_indep_E_nm} we proceed to study the asymptotic independence of $E_{m_j}$ along dyadic sequences $(m_j)$, which we need for the application of \cref{thm_targetindep_0}. Finally, in Subsection \ref{sec_proof_targ_indep_longterm_almost_thm}, we combine all these results to finish the proof of \cref{thm_targetindep_longterm_almost}.

\subsection{Estimates for the measure of $E_{m_{j-1},m_j}$ }
\label{sec_meas_of_E_nm}

\begin{Proposition}
\label{prop_lt_indep_of_E_nm_almost}
Let $(X,\mu,T,\targets)$ be a shrinking target system satisfying \eqref{eqn_targetindependence_longterm_almost}.
Let $m,n,k\in\N$ with $kn\leq m$. Then
\begin{equation}
\label{eqn_lt_indep_of_E_nm_almost}
\mu(E_{kn,m})= \big(1+\oh_{m\to\infty}(1)\big)\, \mu(E_{n,m})^k + \Oh\big(k{F}(m)\mu(B_m)\big).
\end{equation}
\end{Proposition}

For the proof of \cref{prop_lt_indep_of_E_nm_almost} it will be convenient to write $E_{n,m}^*$ for the set 
\begin{equation}
\label{eqn_E_nm_star}
E_{n,m}^*\coloneqq 
\begin{cases}
T^{-{F}(m)}E_{n-{F}(m),m},&\text{if }n>{F}(m),
\\
X,&\text{{otherwise}}.
\end{cases}
\end{equation}
Note that $E_{n,m}^*$ always contains $E_{n,m}$ as a subset. This inclusion follows quickly from the definition of $E_{n,m}$ (see \eqref{eqn_non-hitting_sets} and \eqref{eqn_non-hitting_sets_2}), because
\begin{equation*}
\label{eqn_E_nm_star_inclusion}
E_{n,m}~=~\bigcap_{j=1}^{n} T^{-j} B_m^c~\subset~ \bigcap_{j={F}(m)+1}^{n} T^{-j} B_m^c~=~ T^{-{F}(m)}E_{n-{F}(m),m}~=~ E_{n,m}^*.
\end{equation*}
In general, this inclusion is proper, and the sets $E_{n,m}$ and $E_{n,m}^*$ are not identical. However  they are approximately the same. Indeed, since we are only interested in the case where the quantity ${F}(m)$ is much smaller than $m$, the difference in measure between $E_{n,m}$ and $E_{n,m}^*$ becomes negligible (as we will see in the proofs of \cref{prop_lt_indep_of_E_nm_almost} and \cref{correl_with_error} below). For that reason, we suggest to think of $E_{n,m}^*$ as an approximation of $E_{n,m}$.

The advantage of using $E_{n,m}^*$ over $E_{n,m}$ is that for any $\ell\in\N$ with $\ell\geq n-{F}(m)$ and any set $C\in \G_{\ell,m}$ one has
\begin{equation}
\label{eqn_targetindependence_longterm_almost_for_Enm_star}
\big|\mu(C\cap T^{-\ell}  E_{n,m}^*)- \mu(C)\mu(E_{n,m}^*)\big|\leq \eta(m)\mu(C)\mu(E_{n,m}^*),
\end{equation}
which follows directly from \eqref{eqn_targetindependence_longterm_almost} by choosing $A=C$ and $B=E_{n-{F}(m),m}$.

\begin{proof}[Proof of \cref{prop_lt_indep_of_E_nm_almost}]
Recall that $E_{n,m}=\bigcap_{i=1}^{n} T^{-i} B_m^c$. 
We can split off the first ${F}(m)$ terms in this intersection and thus write $E_{n,m}$ as the intersection of two sets,
\begin{equation}
\label{eqn_rem_0-1}
E_{n,m}=\underbrace{\bigcap_{i=1}^{\min\{n,{F}(m)\}} T^{-i} B_m^c}_{R}~\cap ~{
E_{n,m}^*},
\end{equation}
where $E_{n,m}^*$ is as defined in \eqref{eqn_E_nm_star}.
We can think of $E_{n,m}^*$ as the ``main part'' of $E_{n,m}$ and of $R$ as the ``remainder''.
Since
\begin{equation*}
E_{kn,m}=\bigcap_{j=1}^{kn} T^{-j} B_{m}^c=\bigcap_{j=0}^{k-1} T^{-jn}\left(\bigcap_{i=1}^{n} T^{-i} B_{m}^c\right)=\bigcap_{j=0}^{k-1} T^{-jn} E_{n,m},
\end{equation*}
we can now write
\begin{equation*}
\mu(E_{kn,m})
=\mu\left(\bigcap_{j=0}^{k-1} T^{-jn} E_{n,m}\right)
=\mu\left(R'\cap \bigcap_{j=0}^{k-1} T^{-jn} E_{n,m}^*\right),
\end{equation*}
where $R':=\bigcap_{j=0}^{k-1} T^{-jn}R$.
From this it follows that
\begin{equation}
\label{eqn_lt_indep_of_E_nm_almost_1_upp}
\mu(E_{kn,m})
\leq 
\mu\left(\bigcap_{j=0}^{k-1} T^{-jn} E_{n,m}^*\right),
\end{equation}
which provides us with a suitable upper bound on $\mu(E_{kn,m})$. We also want to find a good lower bound for $\mu(E_{kn,m})$. Observe that the measure of $R$ can trivially be bounded from below by $1-{F}(m)\mu(B_m)$. Therefore, we can bound the measure of $R'$ from below by $1-k\mu(R^c)\geq  1-k{F}(m)\mu(B_m)$.
This gives the estimate
\begin{equation}
\begin{split}
\label{eqn_lt_indep_of_E_nm_almost_1}
\mu(E_{kn,m})&=\mu\left(R'\cap \bigcap_{j=0}^{k-1} T^{-jn} E_{n,m}^*\right)\\
&\geq
\mu\left(\bigcap_{j=0}^{k-1} T^{-jn} E_{n,m}^*\right)-k{F}(m)\mu(B_m).
\end{split}
\end{equation}
To finish the proof, we only have to apply \eqref{eqn_targetindependence_longterm_almost_for_Enm_star} $(k-1)$
times to find that
\begin{equation}
\label{eqn_lt_indep_of_E_nm_almost_2}
\mu\left(\bigcap_{j=0}^{k-1} T^{-jn} E_{n,m}^*\right)\geq\big (1-\eta(m)\big)^{k-1} \mu(E_{n,m}^*)^k
\end{equation}
and
\begin{equation}
\label{eqn_lt_indep_of_E_nm_almost_2_upp}
\mu\left(\bigcap_{j=0}^{k-1} T^{-jn} E_{n,m}^*\right)\leq \big(1+\eta(m)\big)^{k-1} \mu(E_{n,m}^*)^k.
\end{equation}
Finally, since $\mu(E_{n,m}^*)=\mu(E_{n,m})+\Oh\big({F}(m)\mu(B_m)\big)$, we obtain
\begin{equation}
\label{eqn_lt_indep_of_E_nm_almost_3}
\mu(E_{n,m}^*)^k= \mu(E_{n,m})^k +\Oh\big(k{F}(m)\mu(B_m)\big).
\end{equation}
Putting together \eqref{eqn_lt_indep_of_E_nm_almost_1_upp}, \eqref{eqn_lt_indep_of_E_nm_almost_1}, \eqref{eqn_lt_indep_of_E_nm_almost_2}, \eqref{eqn_lt_indep_of_E_nm_almost_2_upp}, and \eqref{eqn_lt_indep_of_E_nm_almost_3} proves \eqref{eqn_lt_indep_of_E_nm_almost}. 
\end{proof}

From \cref{prop_lt_indep_of_E_nm_almost} we can now derive the following corollary.

\begin{Corollary}
\label{cor_meas_of_E_kn_E_knplus1}
For any shrinking target system $(X,\mu,T,\targets)$ that satisfies \eqref{eqn_targetindependence_longterm_almost} and any $m,n,k\in\N$ with $kn\leq m$,
$$
\mu\left(E_{kn,m}\right)= \big(1+\oh_{m\to\infty}(1)\big) \mu(E_{(k+1)n, m})^{\frac{k}{k+1}}+\Oh\left(\big(k {F}(m)\mu(B_m)\big)^{\frac{k}{k+1}}\right).
$$
\end{Corollary}

\begin{proof}
By \cref{prop_lt_indep_of_E_nm_almost}, 
\begin{align*}
\mu(E_{kn, m})
&=
\big(1+\oh(1)\big)\mu(E_{n, m})^k+\Oh\big(k{F}(m)\mu(B_m)\big)
\\
&=
\left(\big(1+\oh(1)\big)\mu(E_{n, m})^{k+1}\right)^{\frac{k}{k+1}}+\Oh\big(k{F}(m)\mu(B_m)\big)
\\
&=
\left(\big(1+\oh(1)\big)\mu(E_{(k+1)n, m})+\Oh\big(k{F}(m)\mu(B_m)\big)\right)^{\frac{k}{k+1}}+\Oh\big(k{F}(m)\mu(B_m)\big)
\\
&=
\big(1+\oh(1)\big)\mu(E_{(k+1)n, m})^{\frac{k}{k+1}}+
\Oh\left(\left(k {F}(m)\mu(B_m)\right)^{\frac{k}{k+1}}\right).
\end{align*}
This finishes the proof.
\end{proof}

\begin{Proposition}
\label{lem_meas_of_E_m_jminus1_m_j}
Let $(X,\mu,T,\targets)$ be a shrinking target system satisfying \eqref{eqn_targetindependence_longterm_almost}. Let $k\geq 2$, and define
$$
m_j\coloneqq k \left\lfloor\frac{(k+1)^{\frac{j}{2}}}{k^{\frac{j}{2}}}\right\rfloor.
$$
Then $\mu(E_{m_{j-1},m_j})=\big(1+\oh_{j\to\infty}(1)\big)\mu(E_{m_{j+1},m_j})^\frac{k}{k+1}+\Oh\left(\left(k{F}(m_j)\mu(B_{m_j})\right)^{\frac{k}{k+1}}\right)$.
\end{Proposition}

\begin{proof}
Set $n_j\coloneqq \left\lfloor\frac{(k+1)^{\frac{j}{2}}}{k^{\frac{j}{2}}}\right\rfloor$.
Then $m_{j-1}=k n_{j-1}$.
Observe that
\begin{align*}
n_{j+1}- \tfrac{k+1}{k} n_{j-1}
&=\left\lfloor\frac{(k+1)^{\frac{j+1}{2}}}{k^{\frac{j+1}{2}}}\right\rfloor - \tfrac{k+1}{k} \left\lfloor\frac{(k+1)^{\frac{j-1}{2}}}{k^{\frac{j-1}{2}}}\right\rfloor
\\
&=-\left\{\frac{(k+1)^{\frac{j+1}{2}}}{k^{\frac{j+1}{2}}}\right\} + \tfrac{k+1}{k} \left\{\frac{(k+1)^{\frac{j-1}{2}}}{k^{\frac{j-1}{2}}}\right\}
=\Oh(1),
\end{align*}
and therefore $kn_{j+1}- (k+1)n_{j-1}=\Oh(k)$.

\smallskip
Observe also that $(k+2)n_{j-1}=\frac{k+2}{k+1}\,\frac{k+1}{k}\,m_{j-1}$, and hence $|m_{j+1}-(k+2)n_{j-1}|$ is bounded from above by $2k$. It follows that $|m_{j+1}-(k+1)n_{j-1}|=\Oh(k)$, and hence
$$
\mu(E_{m_{j+1},m_j})=\mu(E_{(k+1)n_{j-1},m_j})+\Oh\big(k\mu(B_{m_j})\big).
$$
In view of \cref{cor_meas_of_E_kn_E_knplus1} we obtain
\begin{align*}
\mu( & E_{m_{j-1},m_j})
=
\mu(E_{kn_{j-1},m_j})
\\
&=
\big(1+\oh(1)\big)\mu(E_{(k+1)n_{j-1}, m_j})^{\frac{k}{k+1}}+\Oh\left(\left(k{F}(m_j)\mu(B_{m_j})\right)^{\frac{k}{k+1}}\right)
\\
&=
 \big(1+\oh(1)\big)\left(\mu(E_{m_{j+1},m_j})+\Oh\big(k\mu(B_{m_j})\big)\right)^{\frac{k}{k+1}}
+\Oh\left(\left(k{F}(m_j)\mu(B_{m_j})\right)^{\frac{k}{k+1}}\right)
\\
&=
\big(1+\oh(1)\big)\mu(E_{m_{j+1},m_j})^{\frac{k}{k+1}}+\Oh\left(\left(k{F}(m_j)\mu(B_{m_j})\right)^{\frac{k}{k+1}}\right).
\end{align*}
This completes the proof.
\end{proof}

\subsection{Independence of dyadic samples}
\label{sec_indep_E_nm}

\begin{Lemma}\label{correl_with_error}
Let $(X,\mu,T,\targets)$ be a shrinking target system satisfying \eqref{eqn_targetindependence_longterm_almost}. 
For every $s\in\N$ let $m_s$ be a number in $[2^s,2^{s+1})$. 
Then, for all $t>s$,
\begin{equation}
\label{eqn_rc_1}
\begin{split}
\mu(E_{m_s}\cap E_{m_t})\leq \big(1+\oh_{t\to\infty}(1)\big) & \mu(E_{m_s})\mu(E_{m_t})^{1-\frac{2^{s+2}}{2^t}}
\\
&+\Oh\Big(\mu(E_{m_s})\big({F}(m_t)\mu(B_{m_t})\big)^{1-\frac{2^{s+2}}{2^t}}\Big).
\end{split}
\end{equation}
\end{Lemma}

\begin{proof}
It follows from the definition of $E_{m_s}$ and $E_{m_t}$ (see \eqref{eqn_non-hitting_sets} and \eqref{eqn_non-hitting_sets_2}) and the fact that $B_{m_t}\subset B_{m_s}$ that
$$
E_{m_s}\cap E_{m_t}=E_{m_s}\cap T^{-m_s}E_{m_t-m_s,m_t}.
$$
Since $E_{m_t-m_s,m_t}$ is a subset of $E_{m_t-m_s,m_t}^*$, we trivially have
\begin{equation}
\label{eqn_corr_w_e_0}
\mu(E_{m_s}\cap E_{m_t})=\mu(E_{m_s}\cap T^{-m_s}E_{m_t-m_s,m_t})\leq \mu(E_{m_s}\cap T^{-m_s}E_{m_t-m_s,m_t}^*).
\end{equation}
It follows from \eqref{eqn_targetindependence_longterm_almost} that
\begin{equation}
\label{eqn_corr_w_e_1}
\mu(E_{m_s}\cap T^{-m_s}E_{m_t-m_s,m_t}^*)\leq \big(1+\eta(m_t)\big) \mu(E_{m_s})\mu(E_{m_t-m_s,m_t}^*).
\end{equation}
Putting together \eqref{eqn_corr_w_e_0} and \eqref{eqn_corr_w_e_1} we obtain
\begin{equation}
\label{eqn_corr_w_e_1new}
\mu(E_{m_s}\cap E_{m_t})\leq \big(1+\eta(m_t)\big) \mu(E_{m_s})\mu(E_{m_t-m_s,m_t}^*).
\end{equation}

Let $k\coloneqq \lfloor m_t/m_s\rfloor -1$. Since \eqref{eqn_rc_1} trivially holds for $t=s+1$, we can assume that $t>s+1$ and hence $k\neq 0$.
In light of \eqref{eqn_corr_w_e_1new} we see that for the proof of \cref{correl_with_error} it is beneficial to find a good upper bound on the measure of the set $E_{m_t-m_s,m_t}^*$, preferably in terms of the measure of $E_{m_t}$. In order to find such an upper bound, we will first prove the following inequality:
\begin{equation}
\label{eqn_corr_w_e_a}
\frac{\mu(E_{m_t-m_s,m_t}^*)^{\frac{1}{k}}}{
1+\eta(m_t)
}\leq \mu(E_{m_s,m_t}^*).
\end{equation}

Since $km_j\leq m_t-m_s$, the set $E_{m_t-m_s,m_t}^*$ is a subset of $E_{k m_s,m_t}^*$ and hence $\mu(E_{m_t-m_s,m_t}^*)\leq \mu(E_{km_s,m_t}^*)$. Therefore, {\eqref{eqn_corr_w_e_a} is implied by }
\begin{equation}
\label{eqn_corr_w_e_b}
\frac{\mu(E_{km_s,m_t}^*)^{\frac{1}{k}}}{
1+\eta(m_t)
}\leq \mu(E_{m_s,m_t}^*).
\end{equation}
Note that
\begin{eqnarray*}
E_{km_s,m_t}^*
&=&
\bigcap_{i={F}(m_t)+1}^{km_s} T^{-i}B_{m_t}
\\
&=&
\bigcap_{i={F}(m_t)+1}^{m_s} T^{-i}B_{m_t} \cap \bigcap_{i=m_s+1}^{2m_s} T^{-i}B_{m_t}\cap \cdots\cap \bigcap_{i=(k-1)m_s +1}^{k m_s} T^{-i}B_{m_t}.
\end{eqnarray*}
Also observe that for any $\ell\in\{1,\ldots,k-1\}$,
$$
\bigcap_{i=\ell m_s +1}^{(\ell+1)m_s} T^{-i}B_{m_t}
~\subset~
T^{-\ell m_s}\left(\bigcap_{i={F}(m_t)+1}^{m_s} T^{-i}B_{m_t}\right) = T^{-\ell m_s} E_{m_s,m_t}^*.
$$
This proves that
$$
E_{km_s,m_t}^*~\subset~ \bigcap_{\ell=0}^{k-1} T^{-\ell m_s} E_{m_s,m_t}^*.
$$
If we now apply property \eqref{eqn_targetindependence_longterm_almost} to $\mu\left(\bigcap_{\ell=0}^{k-1} T^{-\ell m_s} E_{m_s,m_t}^*\right)$ 
$(k-1)$ 
times, then we see that
$$
\mu(E_{k m_s,m_t}^*)\leq \mu\left( \bigcap_{\ell=0}^{k-1} T^{-\ell m_s} E_{m_s,m_t}^*\right)\leq \big(1+\eta(m_t)\big)^{k-1} \mu(E_{m_s,m_t}^*)^k.
$$
This completes the proof of \eqref{eqn_corr_w_e_b}, and hence also of \eqref{eqn_corr_w_e_a}.

Next, consider the trivial identity
\begin{equation}
\label{eqn_corr_w_e_2}
\mu\left(E_{m_t-m_s,m_t}^*\right)=\tfrac{\mu\left(E_{m_s,m_t}^*)\mu(T^{-m_s}E_{m_t-m_s,m_t}^*\right)}{\mu\left(E_{m_s,m_t}^*\right)}.
\end{equation}
Using \eqref{eqn_corr_w_e_a} we get
\begin{equation}
\label{eqn_corr_w_e_3_2}
\begin{split}
\tfrac{\mu\left(E_{m_s,m_t}^*)\mu(T^{-m_s}E_{m_t-m_s,m_t}^*\right)}{\mu\left(E_{m_s,m_t}^*\right)}
&\leq \big(1+\eta(m_t)\big)\mu(E_{m_s,m_t}^*)\mu(T^{-m_s}E_{m_t-m_s,m_t}^*)^{1-\frac{1}{k}}
\\
&\leq \big(1+\eta(m_t)\big)\Big(\mu(E_{m_s,m_t}^*)\mu(T^{-m_s}E_{m_t-m_s,m_t}^*)\Big)^{1-\frac{1}{k}}.
\end{split}
\end{equation}
Using property \eqref{eqn_targetindependence_longterm_almost} once more, we conclude
\begin{equation}
\label{eqn_corr_w_e_3}
\begin{split}
\big(\mu(E_{m_s,m_t}^*) & \mu(T^{-m_s}E_{m_t-m_s,m_t}^*)\big)^{1-\frac{1}{k}}
\\
& \leq \big(1+\eta(m_t)\big)\mu\left(E_{m_s,m_t}^*\cap T^{-m_s}E_{m_t-m_s,m_t}^*\right)^{1-\frac{1}{k}}.
\end{split}
\end{equation}
{As in the proof of \cref{prop_lt_indep_of_E_nm_almost} (cf.~equation \eqref{eqn_rem_0-1}), we can use the relation
\begin{equation*}
E_{m_t-m_s,m_t}=\bigcap_{i=1}^{\min\{m_t-m_s,{F}(m_t)\}} T^{-i} B_{m_t}^c~\cap ~{
E_{m_t-m_s,m_t}^*}
\end{equation*}
to obtain the estimate
\[
\mu (E_{m_t-m_s,m_t}^*) -\mu\left(\bigcup_{i=1}^{\min\{m_t-m_s,{F}(m_t)\}} T^{-i} B_{m_t}\right) \leq \mu(E_{m_t-m_s,m_t}).
\]
Since $\mu\left(\bigcup_{i=1}^{\min\{m_t-m_s,{F}(m_t)\}} T^{-i} B_{m_t}\right)\hspace{-.1em}\leq \hspace{-.1em}{F}(m_t)\mu(B_{m_t}) $ and $ E_{m_t-m_s,m_t}\hspace{-.1em}\subset \hspace{-.1em} E_{m_t-m_s,m_t}^*$, we have
\[
\mu (E_{m_t-m_s,m_t}^*\triangle E_{m_t-m_s,m_t}) = \Oh\big({F}(m_t)\mu(B_{m_t})\big).
\]
In a similar fashion, one can derive
\[
\mu (E_{m_s,m_t}^*\triangle E_{m_s,m_t}) = \Oh\big({F}(m_t)\mu(B_{m_t})\big).
\]
This gives
} 
\begin{equation*}
\begin{split}
\mu(E_{m_s,m_t}^*\cap T^{-m_s}E_{m_t-m_s,m_t}^*)=& \mu(E_{m_s,m_t}\cap T^{-m_s}E_{m_t-m_s,m_t})+\Oh\big({F}(m_t)\mu(B_{m_t})\big)
\\
=&\mu(E_{m_t})+\Oh\big({F}(m_t)\mu(B_{m_t})\big).
\end{split}
\end{equation*}
It follows that
\begin{equation}
\label{eqn_corr_w_e_4}
\mu(E_{m_s,m_t}^*\cap T^{-m_s}E_{m_t-m_s,m_t}^*)^{1-\frac{1}{k}}=\mu(E_{m_t})^{1-\frac{1}{k}}+\Oh\Big(\big({F}(m_t)\mu(B_{m_t})\big)^{1-\frac{1}{k}}\Big).
\end{equation}
Combining \eqref{eqn_corr_w_e_1new}, \eqref{eqn_corr_w_e_2}, \eqref{eqn_corr_w_e_3_2}, \eqref{eqn_corr_w_e_3} and \eqref{eqn_corr_w_e_4} yields
\begin{equation*}
\mu(E_{m_s}\cap E_{m_t})\leq \big(1+\oh_{t\to\infty}(1)\big)\mu(E_{m_s})\mu(E_{m_t})^{1-\frac{1}{k}}+\Oh\Big(\big(\mu(E_{m_s}){F}(m_t)\mu(B_{m_t})\big)^{1-\frac{1}{k}}\Big).
\end{equation*}
Finally, since $k=\lfloor m_t/m_s\rfloor -1\geq 2^t/2^{s+1}-1\geq 2^t/2^{s+2}$, one has
$$
1-\frac{1}{k}~\geq~ 1-\frac{2^{s+2}}{2^t}.
$$
This finishes the proof.
\end{proof}

\subsection{Proof of \cref{thm_targetindep_longterm_almost}}
\label{sec_proof_targ_indep_longterm_almost_thm}

We need one more lemma before proving \cref{thm_targetindep_longterm_almost}.

\begin{Lemma}
\label{lem_summability}
Suppose that \eqref{eqn_targetindependence_longterm_almost_mathcalH} holds for some $\delta>0$ and all but finitely many $m\in\N$. 
Let $\sigma>1$ and let $(m_j)_{j\in\N}$ be a sequence of natural numbers such that \eq{geom}{{m_{j+1}}/{m_j}\geq\sigma\text{ for all large enough }j\in\N.}
Define $w_j\coloneqq {F}(m_j)\mu(B_{m_j})$. Then
$$\sum_{j\in\N} w_j^{\frac{k}{k+1}}<\infty$$ for all $k\geq 2\delta^{-1}$.
\end{Lemma}
\begin{proof}
In view of \equ{geom}, there exists some $c>0$ such that $m_j\geq c\sigma^{j}$ for all large enough $j\in\N$.
Hence from \eqref{eqn_targetindependence_longterm_almost_mathcalH} we can conclude that
$$
w_j\leq \frac{1}{(\log m_j)^{1+\delta}}
\ll \frac{1}{j^{1+\delta}}\qquad\text{for all but finitely many $j$}.
$$
Since $\sum_{j\in\N}\left(\frac{1}{j^{1+\delta}}\right)^{\frac{k}{k+1}} <\infty$ for all $k$ with $k\geq 2\delta^{-1}$, the claim follows.
\end{proof}

\begin{proof}[Proof of \cref{thm_targetindep_longterm_almost}]
First assume there exists $\epsilon>0$ such that 
$$\sum_{n=1}^\infty \frac{\mu(E_{n})^{1-\epsilon}}{n}<\infty.$$ By assumption, there exists $\delta>0$ such that ${F}(m)\leq \big(\log^{1+\delta}(m)\mu(B_m)\big)^{-1}$ for all but finitely many $m\in\N$. Fix such a $\delta$. Pick now $k\in\N$ with $1/k< \min\{\epsilon,\delta/2\}$.

Next let $(m_j)_{j\in\N}$ be defined as in \cref{lem_meas_of_E_m_jminus1_m_j}, that is, $m_j\coloneqq k \left\lfloor\frac{(k+1)^{\frac{j}{2}}}{k^{\frac{j}{2}}}\right\rfloor$.
It is easy to check that \equ{geom} holds, and, by combining \cref{lem_meas_of_E_m_jminus1_m_j} with \cref{lem_summability} we see that the series
$\sum_{j\in\N}\mu(E_{m_{j-1},m_j})$ converges if and only if so does the series $\sum_{j\in\N}\mu(E_{ m_{j+1},m_j})^\frac{k}{k+1}$.
We now have
\begin{align*}
\sum_{n=1}^\infty \frac{\mu(E_{n})^{1-\epsilon}}{n}
= \sum_{j=1}^\infty ~ \sum_{n\in[m_j,m_{j+1})} \frac{\mu(E_{n})^{1-\epsilon}}{n}
\geq \sum_{j=1}^\infty ~\frac{1}{m_{j+1}} \sum_{n\in[m_j,m_{j+1})}\mu(E_{n})^{1-\epsilon}.
\end{align*}
For any $n$ with $m_j\leq n < m_{j+1}$ we have $\mu(E_n)\geq \mu(E_{m_{j+1},m_j})$. Therefore
\begin{align*}
\sum_{n=1}^\infty \frac{\mu(E_{n})^{1-\epsilon}}{n}
&\geq \sum_{j=1}^\infty ~\frac{1}{m_{j+1}} \sum_{n\in[m_j,m_{j+1})}\mu(E_{m_{j+1},m_j})^{1-\epsilon}
\\
&\geq \sum_{j=1}^\infty ~\frac{m_{j+1}-m_j}{m_{j+1}}\, \mu(E_{m_{j+1},m_j})^{1-\epsilon}
\geq \frac{1}{2k} \sum_{j=1}^\infty \mu(E_{m_{j+1},m_j})^{\frac{k}{k+1}},
\end{align*}
where the last inequality follows because $1-\epsilon \leq \frac{k}{k+1}$ and $\tfrac{m_{j+1}-m_j}{m_{j+1}}\geq \frac{1}{2k}$.
Since $\sum_{n=1}^\infty \frac{\mu(E_{n})^{1-\epsilon}}{n}<\infty$, we conclude that
$\sum_{j=1}^\infty \mu(E_{m_{j+1},m_j})^{\frac{k}{k+1}}<\infty$, and therefore also $\sum_{j\in\N}\mu(E_{m_{j-1},m_j})<\infty$.
In view of \cref{lem_conull}, this implies that $\eah(\targets)$ has full measure, which completes the proof of the first part of \cref{thm_targetindep_longterm_almost}.

\smallskip
For the second part, assume $\sum_{n=1}^\infty\frac{\mu(E_{n})}{n}=\infty$.
For every $j\in\N$ let $m_j$ be a number in $[2^j,2^{j+1})$ that satisfies
$$
\mu(E_{m_j})=\max_{n\in [2^j,2^{j+1})}\mu(E_n).
$$
Then
$$
\sum_{n=1}^\infty\frac{\mu(E_{n})}{n}
= \sum_{j=0}^\infty ~\sum_{n\in[2^j,2^{j+1})}\frac{\mu(E_n)}{n}
\leq \sum_{j=0}^\infty ~\frac{1}{2^j}\sum_{n\in[2^j,2^{j+1})}\mu(E_n)
\leq \sum_{j=0}^\infty \mu(E_{m_j}).
$$
It follows that $\sum_{j=1}^\infty\mu(E_{m_j})=\infty$. 

{According to \cref{correl_with_error}, for all pairs $s,t\in\N$ with $t>s$ we have
\begin{equation}
\label{eqn_rc_1_2_0}
\begin{split}
\mu(E_{m_s}\cap E_{m_t})\leq \big(1+\oh_{t\to\infty}(1)\big) & \mu(E_{m_s})\mu(E_{m_t})^{1-\frac{2^{s+2}}{2^t}}
\\
&+\Oh\Big(\mu(E_{m_s})\big({F}(m_t)\mu(B_{m_t})\big)^{1-\frac{2^{s+2}}{2^t}}\Big).
\end{split}
\end{equation}
Define $v_t\coloneqq (F(m_t)\mu(B_{m_t}))^{1-2^{s+2-t}}$, and set $$w_{0,j}\coloneqq F(m_{2j})\mu(B_{m_{2j}})\text{ and }w_{1,j}\coloneqq F(m_{2j-1})\mu(B_{m_{2j-1}}).$$
Then we have
\begin{align*}
\sum_{t\in \N} v_t 
&\,=\, \sum_{j\in \N} v_{2j}  \,+\,\sum_{j\in \N} v_{2j-1}
\,=\, \sum_{j\in \N} w_{0,j}^{1-2^{s+2-2j}} \,+\,\sum_{j\in \N} w_{1,j}^{1-2^{s+3-2j}}.
\end{align*}
Since $w_{0,j}= F(m_{2j})\mu(B_{m_{2j}})$ and ${m_{2(j+1)}}/{m_{2j}}\geq2$ for all $j\in\N$, it follows from \cref{lem_summability} that 
\[
\sum_{j\in \N} w_{0,j}^{1-2^{s+2-2j}} <\infty.
\]
In an analogous way, one can show that
\[
\sum_{j\in \N} w_{1,j}^{1-2^{s+3-2j}} <\infty.
\]
Therefore, we have $\sum_{t\in\N} v_t<\infty$,
which, in combination with \eqref{eqn_rc_1_2_0}, proves that
\[
\mu(E_{m_s}\cap E_{m_t})\leq \big(1+\oh_{t\to\infty}(1)\big) \mu(E_{m_s})\mu(E_{m_t})^{1-2^{s-t+2}} + \Oh\big(\mu(E_{m_s}) v_t\big)
\]
where $(v_t)_{t\in\N}$ satisfies $\sum_{t\in\N} v_t<\infty$. 
}
Thus, by \cref{thm_targetindep_0}, we conclude that $\mu\big(\eah(\targets)\big)$ is not equal to $1$. Since $\mu\big(\eah(\targets)\big)$ is essentially invariant (see \cref{prop_01law}), we must have $\mu\big(\eah(\targets)\big)=0$, which finishes the proof.   
\end{proof}

\section{Shrinking target systems with independent targets}
\label{sec_ind_targ}

Let us now show how \cref{thm_targetindep_1} and the corresponding \cref{cor_targetindep_1} follow from the results we have obtained so far.

\begin{proof}[Proof of \cref{thm_targetindep_1}]
If follows immediately from property \eqref{eqn_targetindependence} and the definition of $E_m$ (see \eqref{eqn_E_m}) that
\begin{equation*}
\mu(E_m) = \big(1-\mu(B_m)\big)^m.
\end{equation*}
Hence \cref{thm_targetindep_1} follows from \cref{thm_targetindep_longterm_almost}.
\end{proof}

\begin{proof}[Proof of \cref{cor_targetindep_1}]
First assume for all but finitely many $m\in\N$ that
$$
\mu(B_m)\geq \frac{C \log\log m}{m}.
$$

Choose any $b \in (1,C)$. 
Using the inequality $(1+x)\leq e^x$, which holds for all real numbers $x$, we obtain (with $x=- (C\log(k\log b/2))/\lfloor b^{k} \rfloor$) that for all sufficiently large $k$,
\begin{align*}
\left(1-\frac{C \log\log \lfloor b^k\rfloor }{{\lfloor b^k\rfloor}}\right)^{\lfloor b^{k} \rfloor}
&\leq
\left(1-\frac{C \log\log(b^{k/2}) }{{\lfloor b^k\rfloor}}\right)^{\lfloor b^{k} \rfloor}
\\
&= 
\left(1-\frac{C \log\left(\frac{k \log b}{2}\right) }{{\lfloor b^k\rfloor}}\right)^{\lfloor b^{k} \rfloor}
\\
&\leq e^{-C \log\left(\frac{k \log b}{2}\right)}=\frac{e^{-C\log\left(\frac{\log b}{2}\right)}}{k^{{C}}}.
\end{align*}
Then
\begin{align*}
\sum_{m=1}^{\infty}\frac{(1-\mu(B_m))^{m(1-\epsilon)}}{m}
&\leq \sum_{m=1}^{\infty}\frac{\left(1-\frac{C\log\log m}{m}\right)^{m(1-\epsilon)}}{m}
\\
&\leq \sum_{k=1}^\infty\left(1-\frac{C \log\log \lfloor b^k\rfloor }{{\lfloor b^k\rfloor}}\right)^{\lfloor b^{k} \rfloor(1-\epsilon)}
~=~\Oh\left( \sum_{k=1}^\infty \frac{1}{k^{{C}(1-\epsilon)}}\right).
\end{align*}
Since $\sum_{k=1}^\infty \frac{1}{k^{{C}(1-\epsilon)}} <\infty$ for sufficiently small $\epsilon$, it follows from \cref{thm_targetindep_1} that $\eah(\targets)$ has full measure.

\smallskip
The second part follows from an analogous calculation where instead of the inequality $(1+x)\leq e^x$ one uses the inequality $(1+x) \geq e^{x-x^2}$, which holds for all $x\in(-1/2,0]$. Indeed,
\begin{align*}
\sum_{m=1}^{\infty}\frac{\mu(E_m)}{m}
&=
\sum_{m=1}^{\infty}\frac{(1-\mu(B_m))^{m}}{m}
\geq \sum_{m=1}^{\infty}\frac{\left(1-\frac{\log\log m}{m}\right)^{m}}{m}
\\
&\geq \frac{1}{2}\sum_{k=1}^{\infty}\left(1-\frac{\log k}{2^k}\right)^{2^k}
\geq \frac{1}{2}\sum_{k=1}^{\infty}\left(e^{-\frac{\log k}{2^k}-\frac{\log^2 k}{2^{2k}}}\right)^{2^k}
\\
&
= \frac{1}{2}\sum_{k=1}^\infty \frac{1}{k}e^{-\frac{\log^2 k}{2^{k}}} ~=~\infty.
\end{align*}
Therefore, by \cref{thm_targetindep_1}, $\eah(\targets)$ has zero measure.
\end{proof}

\section{Bernoulli schemes and a proof of \cref{thm_bernoulli}}
\label{sec_bernoulli_proof}

In this section we give a proof of \cref{thm_bernoulli}. Let $(r_n)_{n\in\N}$ and $(X
,\mu
,T
,\targets
)$ be as in Subsection \ref{sec_bernoulli}.
Given a point $x\in X
=\{0,1\}^{\N\cup\{0\}}$ we denote by $x[1,\dots,n]$ the word $x[1]x[2]\ldots x[n]$.

In order to derive \cref{thm_bernoulli} from \cref{thm_targetindep_longterm_almost} we first need to understand the measure of the set $E_n=\bigcap_{j=1}^{n}T^{-j}B_n^c$. Note that $E_n$ consists exactly of all the points $x\in\{0,1\}^{\N\cup\{0\}}$ with the property that the word $x[1,n+r_n]$ does not contain $r_n$ consecutive zeros. To estimate $\mu(E_n)$, it will therefore be convenient to beforehand estimate the average number of zeros in $x[1,n+r_n]$. For each $n\geq 1$ and $x\in X$, let $V_n(x) \coloneqq \max \{\textrm{number of consecutive zeros in }x[1,\dots,n]\}$.
Let ${\log_2} x\coloneqq \frac{\log x }{\log2 }$. Our main tool is the following estimate from \cite{FS}.


\begin{Proposition}[{\cite[Proposition V.1.]{FS}}]\label{prop_dist}
Let $a(n) \coloneqq 2^{\{{\log_2}{n}\}}$, where we use $\{.\}$ to denote the fractional part of a real number. One has
$$\mu\big(V_{n}<\lfloor {\log_2}{n}\rfloor + h\big) = \exp\big(-a(n)2^{-h-1}\big)+\Oh\left(\frac{\log{n}}{\sqrt{n}}\right). $$
\end{Proposition}

Note that
\eq{en}{
{E_{n,m}=
\{x\in X: V_{n+r_m}(x)<r_m
\}.}
}
Using this, we can get the first order asymptotics for $\mu(E_{n,m})$.

\begin{Theorem}
\label{thm_measure_of_Enm}
One has $$\mu(E_{n,m}) = \exp\big(-\tfrac{n}{2}\mu(B_{m})\big)(1+\oh_{m\to\infty}(1))+\Oh\left(\frac{\log{n}}{\sqrt{n}}\right).$$ 
\end{Theorem}

\begin{proof} We will write $h_{n,m}\coloneqq r_m-\lfloor {\log_2}(n+r_m)\rfloor$. 
In view of \equ{en},
\begin{align*}
\mu(E_{n,m})&= \exp\left(-a(n+r_m)2^{-h_{n,m}-1}\right) + \Oh\left(\frac{\log(n+r_m)}{\sqrt{n+r_m}}\right).
\end{align*}
We can replace $\Oh\left(\frac{\log(n+r_m)}{\sqrt{n+r_m}}\right)$ with $\Oh\left(\frac{\log n }{\sqrt{n}}\right)$. Thus,
\begin{align*}
\mu(E_{n,m})&= \exp\big(-a(n+r_m)2^{-h_{n,m}-1}\big) + \Oh\left(\frac{\log n }{\sqrt{n}}\right)
\\
&= \exp\big(-2^{\{{\log_2}(n+r_m)\}}2^{-(r_m-\lfloor {\log_2}(n+r_m)\rfloor)-1}\big) + \Oh\left(\frac{\log n }{\sqrt{n}}\right)
\\
&= \exp\big(-2^{{\log_2}(n+r_m)-r_m-1}\big) + {\Oh\left(\frac{\log n }{\sqrt{n}}\right)}
\\
&= \exp\big(-(n+r_m)2^{-r_m-1}\big) + \Oh\left(\frac{\log n }{\sqrt{n}}\right)
\\
&= \exp\big(-n2^{-r_m-1}\big)\big(1+\oh_{m\to\infty}(1)\big) + \Oh\left(\frac{\log n }{\sqrt{n}}\right)
\\
&= \exp\left(-\frac{n}{2}\mu(B_m)\right)(1+\oh_{m\to\infty}(1)) + {\Oh\left(\frac{\log n }{\sqrt{n}}\right)}.
\end{align*}
From this the claim follows.
\end{proof}

Choosing $n=m$ in \cref{thm_measure_of_Enm} yields the following corollary.

\begin{Corollary}
\label{thm_measure_of_Em}
One has $$\mu(E_{m}) = \exp\big(-\tfrac{m}{2}\mu(B_{m})\big)\big(1+\oh_{m\to\infty}(1)\big)+\Oh\left(\frac{\log{m}}{\sqrt{m}}\right).$$ 
\end{Corollary}

\begin{Remark}
\label{rem_enm}
\cref{thm_measure_of_Enm} can also be useful to estimate the measure of sets of the from $E_{m_{j-1},m_j}$, which are of interest because of \cref{lem_conull}.
For the proof of \cref{thm_bernoulli}, which we will present at the end of this section, we are particularly interested in the case where
$$
m_j=\lfloor b^j \rfloor
$$
for some $b>1$. In this case, it follows from \cref{thm_measure_of_Enm} that
\begin{equation}
\label{eqn_emj-1mj_estimate}
\mu(E_{m_{j-1},m_j})=
\exp\big(-\tfrac{m_{j-1}}{2}\mu(B_{m_j})\big)\big(1+\oh_{m\to\infty}(1)\big)+\Oh\left(\frac{j}{b^{\frac{j}{2}}}\right).
\end{equation}
Since $m_{j-1}\geq\frac{m_j}{c}$ for all but finitely many $j$ as long as $c>b$, we deduce from \eqref{eqn_emj-1mj_estimate} that
\begin{equation}
\label{eqn_emj-1mj_estimate_2}
\mu(E_{m_{j-1},m_j})\leq
\exp\big(-\tfrac{m_{j}}{2c}\mu(B_{m_j})\big)\big(1+\oh_{m\to\infty}(1)\big)+\Oh\left(\frac{j}{b^{\frac{j}{2}}}\right)
\end{equation}
for all $c>b$.
\end{Remark}

\begin{Theorem}
\label{thm_bernoulli_intermediate}
Let $(X
,\mu
,T
,\targets
)$ be a shrinking target system.
If there exists $\epsilon>0$ such that
$$
\sum_{m=1}^\infty \frac{\big(1-\mu(B_{m})\big)^{\frac{m(1-\epsilon)}{2}}}{m} <\infty,
$$
then there exists $\epsilon'>0$ such that
$$
\sum_{m=1}^\infty \frac{\exp\big(-\frac{m}{2}\mu(B_{m})\big)^{1-\epsilon'}}{m}<\infty.
$$
Also, if
$$
\sum_{m=1}^\infty \frac{\big(1-\mu(B_{m})\big)^{\frac{m}{2}}}{m} =\infty
$$
then
$$
\sum_{m=1}^\infty \frac{\exp\big(-\frac{m}{2}\mu(B_{m})\big)}{m}=\infty.
$$
\end{Theorem}

\begin{proof}
Using the basic inequality $(1+x)^r\leq \exp(rx)$ (which holds for all $x\geq -1$ and $r>0$ and follows readily from $\ln(1+x) \leq x$) it is straightforward to show that
$$
\big(1-\mu(B_{m})\big)^{\frac{m}{2}}\leq \exp\big(-\tfrac{m}{2}\mu(B_{m})\big).
$$
From this the implication
$$
\sum_{m=1}^\infty \frac{\big(1-\mu(B_{m})\big)^{\frac{m}{2}}}{m} =\infty
\implies
\sum_{m=1}^\infty \frac{\exp\big(-\frac{m}{2}\mu(B_{m})\big)}{m}=\infty
$$
follows.

For the other implication, we can use the inequality
\begin{align}
\label{eqn_ineq_1}
(1+x)^{\delta r}\,\geq\, \exp(rx),
\end{align} 
which holds for every $\delta<1$, all $r>0$, and all non-positive $x$ that are sufficiently close to $0$, where the closeness to $0$ depends only on $\delta$ but not on $r$. The validity of \eqref{eqn_ineq_1} follows from $\ln(1+x)\geq \delta^{-1} x$ by exponentiating both sides and then raising to the power of $r$. Hence, for $\delta=\frac{1-\epsilon}{1-\epsilon'}$ we get 
{that for all but finitely many $m$}
$$
\big(1-\mu(B_{m})\big)^{\frac{m(1-\epsilon)}{2}}\geq \exp\big(-\tfrac{m}{2}\mu(B_{m})\big)^{1-\epsilon'},
$$
where $\epsilon'>0$ can be any number that is strictly smaller than $\epsilon$.
This implies
$$
\sum_{m=1}^\infty \frac{\big(1-\mu(B_{m})\big)^{\frac{m(1-\epsilon)}{2}}}{m} <\infty
~\implies~ \sum_{m=1}^\infty \frac{\exp\big(-\frac{m}{2}\mu(B_{m})\big)^{1-\epsilon'}}{m}<\infty.
$$
\end{proof}

Next, we will present a proof of \cref{thm_bernoulli}. For the reader's benefit, let us briefly outline the main ideas behind the proof first.
Recall that by assumption either \eqref{loglog} or \eqref{tau} are satisfied.
We will show below that  \eqref{tau} forces conditions \eqref{eqn_targetindependence_longterm_almost} and \eqref{eqn_targetindependence_longterm_almost_mathcalH} to be satisfied for an appropriate choice of $F$ and $\eta$, which will allow us to derive the conclusion of \cref{thm_bernoulli} from \cref{thm_targetindep_longterm_almost}. On the other hand, under the assumption \eqref{loglog} we cannot guarantee that \eqref{eqn_targetindependence_longterm_almost_mathcalH} is satisfied, because the measure of the targets $B_m$ might not shrink sufficiently fast. In this case, instead of using \cref{thm_targetindep_longterm_almost}, our argument will build on \cref{rem_enm} together with \cref{lem_conull}.

\begin{proof}[Proof of \cref{thm_bernoulli}]
Let us first deal with {\eqref{loglog}}.
Let $c$ be such that $1<c<D/2$, and define $\eta\coloneqq \frac D4-\frac c2$. 
Since $\mu(B_m)\geq \frac{2(c+2\eta)\log\log m}{m}$ for all but finitely many $m$, it follows that $\frac{m}{2}\mu(B_m)\geq (c+2\eta)\log\log m\geq (c+\eta)\log\left(\frac{\log m}{\log c}\right)$. Applying the map $x\mapsto\exp(-x/c)$ to both sides of this inequality yields
\begin{equation}
\label{eqn_q1}
\exp\left(-\frac{m}{2c}\mu(B_m)\right)\leq \left(\frac{\log c}{\log m
}\right)^{1+\frac{\eta}{c}}
\end{equation}
for all but finitely many $m$. Let $b\in (1,c)$ be arbitrary. 
After substituting $m_j=\lfloor b^j\rfloor$ for $m$ in \eqref{eqn_q1}, we are left with
$$
\exp\left(-\frac{m_j}{2}\mu(B_{m_j})\right)\leq \frac{1}{j^{1+\frac{\eta}{c}}}.
$$
Combining this with \eqref{eqn_emj-1mj_estimate_2} shows $\sum_{j\in\N}\mu(E_{m_{j-1},m_j})<\infty$. In light of \cref{lem_conull}, this proves that $\eah(\targets)$ has full measure.

\smallskip
Next, we deal with \eqref{tau}.
Pick $F(m)=r_m$ and $\eta=0$.
Choose $\delta>0$ sufficiently small such that 
 $\tau\geq\frac{1+\delta}{1-\delta}$.
Since $$\mu(B_m)\log_2\left(\frac{1}{\mu(B_m)}\right)\leq \mu(B_m)^{1-\delta}$$ for all but finitely many $m\in\N$ (because $\lim_{m\to\infty}\mu(B_m)=0$), we deduce that
$$
F(m) \mu(B_m) = \mu(B_m)\log_2\left(\tfrac{1}{\mu(B_m)}\right)\leq \mu(B_m)^{1-\delta} \leq \left(\frac{1}{(\log m)^\tau }\right)^{\frac{1}{1-\delta}}\leq \frac{1}{(\log m)^{1+\delta}}.
$$
Hence $F$ satisfies \eqref{eqn_targetindependence_longterm_almost_mathcalH}.
By construction, the shrinking target system also satisfies \eqref{eqn_targetindependence_longterm_almost}.
In light of \cref{thm_measure_of_Em} 
{there exists a constant $C>1$ such that
\begin{equation}
\label{eqn_ineq_010}
C^{-1} \exp\big(-\tfrac{m}{2}\mu(B_{m})\big)-C\, 
\frac{\log{m}}{\sqrt{m}}
\leq \mu(E_{m}) \leq C \exp\big(-\tfrac{m}{2}\mu(B_{m})\big)+C\, 
\frac{\log{m}}{\sqrt{m}}
\end{equation}
holds for all but finitely many $m$.
Since
$$
\sum_{m\in\N}\frac{\log{m}}{m\sqrt{m}}< \infty,
$$
we conclude that
}
\begin{equation}
\label{eqn_ser_conv_eqiv_1}
\sum_{m=1}^\infty\frac{\mu(E_m)}{m}=\infty
~\iff~ \sum_{m=1}^\infty \frac{\exp\big(-\frac{m}{2}\mu(B_{m})\big)}{m}=\infty.
\end{equation}
{Then, using the inequalities $x^{1-\epsilon}-y^{1-\epsilon}\leq (x-y)^{1-\epsilon}$ (which holds for all $x\geq y\geq 0$) and $(x+y)^{1-\epsilon}\leq x^{1-\epsilon}+y^{1-\epsilon}$ (which holds for all $x,y\geq 0$), it follows from \eqref{eqn_ineq_010} that
\begin{align*}
C^{-1} \exp\big(-\tfrac{m}{2}\mu(B_{m})\big)^{1-\epsilon}-C & \left(\frac{\log{m}}{\sqrt{m}}\right)^{1-\epsilon} \leq \mu(E_{m})^{1-\epsilon}
\\
&\leq C \exp\big(-\tfrac{m}{2}\mu(B_{m})\big)^{1-\epsilon}+C\left(\frac{\log{m}}{\sqrt{m}}\right)^{1-\epsilon}.
\end{align*}
Combining this with
$$
\sum_{m\in\N} \frac{\left(\frac{\log{m}}{\sqrt{m}}\right)^{1-\epsilon}}{m} < \infty 
$$
shows that we also have}
\begin{equation}
\label{eqn_ser_conv_eqiv_2}
\sum_{m=1}^\infty\frac{\mu(E_m)^{1-\epsilon}}{m}<\infty
~\iff~ \sum_{m=1}^\infty \frac{\exp\big(-\frac{m}{2}\mu(B_{m})\big)^{1-\epsilon}}{m}<\infty.
\end{equation}
Hence{, in light of \eqref{eqn_ser_conv_eqiv_1} and \eqref{eqn_ser_conv_eqiv_2},} \cref{thm_bernoulli} follows directly from \cref{thm_targetindep_longterm_almost} together with \cref{thm_bernoulli_intermediate}. 
\end{proof}

{\cref{cor_targetindep_2} can be derived from \cref{thm_bernoulli} the same way that \cref{cor_targetindep_1} was derived from \cref{thm_targetindep_1}. Therefore we omit its proof.}

\section{The Gau{\ss} map and the Gau{\ss} measure}
\label{sec_gauss_proof}

In this section let $(X,\mu,T,\targets)$ denote the shrinking target system considered in Subsection \ref{sec_cft}, where $X$ is the interval $[0,1]$, $T\colon[0,1]\to[0,1]$ is the Gau{\ss} map, $\mu$ is the Gau{\ss} measure, and $\targets=\{B_1\supset B_2\supset\ldots\}$ 
{is defined by \eqref{eqn_def_Bm_gauss}}.

We begin by showing that for this shrinking target system condition \eqref{eqn_targetindependence_longterm_almost} holds for any ${F}(m)$ that satisfies \eqref{eqn_targetindependence_longterm_almost_mathcalH} and $\eta(m)=\Oh\left(\big(-C{\sqrt{{F}(m)}}\,\big)\right)$ for some universal constant $C>0$. The following result of Phillipp will be crucial for making this deduction.

\begin{Lemma}[\rm{\cite{Ph}}]
\label{lem_wp_lem2}
There exists a constant $\lambda\in(0,1)$ such that for {all $k,n\in\N$, all sets of the form 
\begin{equation}
\label{eqn_form}
\begin{aligned}
A=\big\{[a_1,a_2,\ldots]: a_1=r_1,\ldots,a_n=r_n\big\},\\ 
\text{ where $r_1,\ldots,r_n\in\N$ are arbitrary,}\quad
\end{aligned}
\end{equation}}
and all measurable sets $B\subset [0,1]$ one has
$$
\mu(A\cap T^{-n-k} B)= \mu(A) \mu(B) \left(1+\Oh\big(\lambda^{\sqrt{k}}\big)\right).
$$
\end{Lemma}


{From \cref{lem_wp_lem2} we can derive the following corollary.
\begin{Corollary}
\label{co_wp_lem2}
Let $\Theta_{n}$ be the $\sigma$-algebra on $[0,1]$ generated by all sets {of the form \eqref{eqn_form}}.
There exists a constant $\lambda\in(0,1)$ such that for all $k,n\in\N$, all $A\in\Theta_n$, and all measurable sets $B\subset [0,1]$ one has
$$
\mu(A\cap T^{-n-k} B)= \mu(A) \mu(B) \left(1+\Oh\big(\lambda^{\sqrt{k}}\big)\right).
$$
\end{Corollary}
\begin{proof}
Fix $k,n\in\N$. Any set $A\in\Theta_n$ can be written (up to null sets) as a disjoint union 
$$
A=\bigcup_{j\in J} A_j
$$
where $J$ is either finite or countably infinite, and for each $j\in J$ the set $A_j$ is of the form
$$
A_j = \big\{[a_1,a_2,\ldots]: a_1=r_1^{(j)},\ldots,a_n=r_n^{(j)}\big\}
$$
for some $r_1^{(j)},\ldots,r_n^{(j)}\in\N$.
In light of \cref{lem_wp_lem2} we have
$$
\mu(A_j\cap T^{-n-k} B)= \mu(A_j) \mu(B) \left(1+\Oh\big(\lambda^{\sqrt{k}}\big)\right)
$$
for every $j\in J$.
Moreover, seeing that $(A_j)_{j\in J}$ are pairwise disjoint and $J$ is countable, we have
$$
\mu(A) =\sum_{j\in J} \mu(A_j)\quad\text{and}\quad \mu(A\cap T^{-n-k} B)=\sum_{j\in J} \mu(A_j\cap T^{-n-k} B) .
$$
The claim now follows.
\end{proof}

{Recall that given $n\leq m$  and a shrinking target $\mathcal{B}$ we have defined the algebra $\G_{n,m}$ by \eqref{eqn_thetanm}.}
Since the $\sigma$-algebra $\Theta_m$ appearing in \cref{co_wp_lem2} contains $\G_{n,m}$ 
as a sub-algebra, it follows from the conclusion of \cref{co_wp_lem2} that \eqref{eqn_targetindependence_longterm_almost} holds for any ${F}(m)$ that satisfies \eqref{eqn_targetindependence_longterm_almost_mathcalH} and $\eta(m)=\Oh\left(\big(-C{\sqrt{{F}(m)}}\,\big)\right)$ for some universal constant $C>0$.}



\begin{Proposition}
\label{prop_meas_of_E_nm_gauss}
{For all $n\leq m \in\N$ we have
\begin{equation}
\label{eqn_meas_of_E_nm_gauss_11}
\left(1-\frac{2}{k_m}\right)^n\leq \mu(E_{n,m})\leq \left(1-\frac{1}{k_m+1}\right)^n.
\end{equation}}
\end{Proposition}

{In what follows, for any $t\in [0,1]$ let $M_t:[0,1]\to[0,1]$ denote the map $M_t(x)=tx$.
For the proof of \cref{prop_meas_of_E_nm_gauss} we will need the following lemma.
\begin{Lemma}
\label{lem_hl_7_j}
For all $k\in\N$ and all measurable $A\subset [0,1]$ we have
\begin{equation}
\label{eqn_sdr_1}
\mu\Big(\big[0,\tfrac{1}{k}\big]\cap T^{-1}A \Big)
=
\mu\big(M_{{1}/{k}}A\big).
\end{equation}
\end{Lemma}
}

\begin{proof}
{Note that both $A\mapsto \mu\big([0,\tfrac{1}{k}]\cap T^{-1}A \big)$ and $A\mapsto \mu(M_{{1}/{k}}A)$ are Borel measures on $[0,1]$. To show that two Borel measures coincide it suffices to verify equality for a family of sets that generate the Borel $\sigma$-algebra. In particular, instead of proving \eqref{eqn_sdr_1} for all measurable sets $A\subset [0,1]$, it suffices to show
\begin{equation}
\label{eqn_sdr_2}
\mu\Big(\big[0,\tfrac{1}{k}\big]\cap T^{-1}[0,s]\Big)
=
\mu\big(\big[0,\tfrac{s}{k}\big]\big)
\end{equation}
for all $0\leq s\leq 1$.}

{ 
Note that $T^{-1}[0,s]\hspace{-.09em}=\hspace{-.09em}\bigcup_{n\in\N}\big[\tfrac{1}{n+s},\tfrac{1}{n}\big]$ and hence $\big[0,\tfrac{1}{k}\big]\cap T^{-1}[0,s]\hspace{-.09em}=\hspace{-.09em}\bigcup_{n\geq k}\big[\tfrac{1}{n+s},\tfrac{1}{n}\big]$. We conclude that
\[
\mu\Big(\big[0,\tfrac{1}{k}\big]\cap T^{-1}[0,s]\Big)=\sum_{n=k}^\infty \mu\big(\big[\tfrac{1}{n+s},\tfrac{1}{n}\big]\big).
\]
{Therefore} we have
\begin{align*}
\mu\big(\big[\tfrac{1}{n+s},\tfrac{1}{n}\big]\big)=\int_{\frac{1}{n+s}}^{\frac{1}{n}}\frac{\d x}{1+x}=\log\big(1+\tfrac{1}{n}\big)-\log\big(1+\tfrac{1}{n+s}\big)
=
\log\biggl(\frac{1+\tfrac{1}{n}}{1+\tfrac{1}{n+s}}\biggr).
\end{align*}
An analogous calculation yields
\begin{align*}
\mu\big(\big[\tfrac{s}{n+1},\tfrac{s}{n}\big]\big)=\int_{\frac{s}{n+1}}^{\frac{s}{n}}\frac{\d x}{1+x}=
\log\biggl(\frac{1+\tfrac{s}{n}}{1+\tfrac{s}{n+1}}\biggr).
\end{align*}
Since
\[
\frac{1+\tfrac{1}{n}}{1+\tfrac{1}{n+s}}
=\frac{\tfrac{n+1}{n}}{\tfrac{n+s+1}{n+s}}
=\frac{\tfrac{n+s}{n}}{\tfrac{n+s+1}{n+1}}
=\frac{1+\tfrac{s}{n}}{1+\tfrac{s}{n+1}},
\]
it follows that
\[
\mu\big(\big[\tfrac{1}{n+s},\tfrac{1}{n}\big]\big)=\mu\big(\big[\tfrac{s}{n+1},\tfrac{s}{n}\big]\big).
\]
Summing over $n\geq k$ finishes the proof of \eqref{eqn_sdr_2}.}
\end{proof}

\begin{proof}[Proof of \cref{prop_meas_of_E_nm_gauss}]
Consider the set $\tilde{E}_{n,m}\coloneqq \bigcap_{i=0}^{n-1} T^{-i}B_m^c$. Since $E_{n,m}=T^{-1}\tilde{E}_{n,m}$, and since the Gau{\ss} measure $\mu$ is invariant under $T$,
{it follows that $\mu(E_{n,m}) =  \mu(\tilde E_{n,m})$. Thus}
it suffices to prove \eqref{eqn_meas_of_E_nm_gauss_11} with $E_{n,m}$ replaced by $\tilde{E}_{n,m}$.
{We will also make use of the fact that for any $a\in[0,1]$ and any measurable $A\subset [0,a]$ one has
\begin{equation}
\label{eqn_gauss_Lebesgue_inequ}
\frac{\lambda(A)}{(1+a)(\log 2)}\leq \mu(A)\leq \frac{\lambda(A)}{\log 2}.
\end{equation}}

{Let us prove \eqref{eqn_meas_of_E_nm_gauss_11} (with $E_{n,m}$ replaced by $\tilde{E}_{n,m}$) by induction on $n$. If $n=1$ then \eqref{eqn_meas_of_E_nm_gauss_11} says
\[
1-\frac{2}{k_m}
\leq 1-\mu(B_m)\leq 
1-\frac{1}{k_m+1}
.
\]
The validity of this statement for all $m\in\N$ is straightforward to check using $\lambda(B_m)=1/k_m$ and \eqref{eqn_gauss_Lebesgue_inequ}.} 

{Next, let $n\geq 1$ and assume \eqref{eqn_meas_of_E_nm_gauss_11} has already been proven for $n$.
Our goal is to show that it also holds for $n+1$.
Using $\tilde{E}_{n+1,m}=B_m^c\cap T^{-1} \tilde{E}_{n,m}$, we get
\begin{align*}
\mu(\tilde{E}_{n+1,m}) &= \mu(B_m^c\cap T^{-1} \tilde{E}_{n,m})
\\
&=\mu(T^{-1}\tilde{E}_{n,m})-\mu(B_m\cap T^{-1} \tilde{E}_{n,m})
\\
&=\mu(\tilde{E}_{n,m})-\mu(B_m\cap T^{-1} \tilde{E}_{n,m})
\\
&=\mu(\tilde{E}_{n,m})-\mu(M_{{1}/{k_m}} \tilde{E}_{n,m}),
\end{align*}
where the last equality follows from \cref{lem_hl_7_j}.
Using \eqref{eqn_gauss_Lebesgue_inequ}, we can derive an upper bound on the last term of the above displayed equation as follows:
\begin{align*}
\mu(\tilde{E}_{n,m})\hspace{-.1em}-\hspace{-.1em}\mu(M_{{1}/{k_m}} \tilde{E}_{n,m})\hspace{-.1em}
&\leq \mu(\tilde{E}_{n,m})-\frac{\lambda(M_{{1}/{k_m}} \tilde{E}_{n,m})}{(1+\frac{1}{k_m})(\log 2)}
\\
&=
\mu(\tilde{E}_{n,m})-\frac{\frac{1}{k_m}\lambda( \tilde{E}_{n,m})}{(1+\frac{1}{k_m})(\log 2)}
\\
&=
\mu(\tilde{E}_{n,m})-
\frac{1}{k_m+1} \cdot\frac{\lambda( \tilde{E}_{n,m})}{(\log 2)}
\\
&\leq
\mu(\tilde{E}_{n,m})-
\frac{1}{k_m+1}\cdot \mu( \tilde{E}_{n,m})
\hspace{-.1em}=\hspace{-.1em}
\mu(\tilde{E}_{n,m})\hspace{-.1em}\left(1-
\frac{1}{k_m+1}\right).
\end{align*}
We conclude that 
\begin{equation}
\label{eqn_E-tilde_upper_bound}
\mu(\tilde{E}_{n+1,m})\leq \mu(\tilde{E}_{n,m})\left(1-
\frac{1}{k_m+1}\right).
\end{equation}
In a similar fashion, we can also obtain a lower bound:
\begin{align*}
\mu(\tilde{E}_{n,m})-\mu(M_{{1}/{k_m}} \tilde{E}_{n,m})
&\geq
\mu(\tilde{E}_{n,m})-\frac{\lambda(M_{{1}/{k_m}} \tilde{E}_{n,m})}{(\log 2)}
\\
&=
\mu(\tilde{E}_{n,m})-\frac{\frac{1}{k_m}\lambda( \tilde{E}_{n,m})}{(\log 2)}
\\
&=
\mu(\tilde{E}_{n,m})-\frac{2}{k_m}\cdot\frac{\lambda(\tilde{E}_{n,m})}{2(\log 2)}
\\
&\geq
\mu(\tilde{E}_{n,m})-\frac{2}{k_m}\cdot\mu(\tilde{E}_{n,m})
=\mu(\tilde{E}_{n,m})\left(1-
\frac{2}{k_m}\right).
\end{align*}
It follows that
\begin{equation}
\label{eqn_E-tilde_lower_bound}
\mu(\tilde{E}_{n+1,m})\geq \mu(\tilde{E}_{n,m})\left(1-
\frac{2}{k_m}\right).
\end{equation}
Combining \eqref{eqn_E-tilde_upper_bound} and \eqref{eqn_E-tilde_lower_bound} with the induction hypothesis finishes the proof.}
\end{proof}

\begin{Corollary}
\label{cor_meas_of_E_nm_gauss}
{For every $\epsilon>0$ the inequalities
$$
\big(1-\mu(B_m)\big)^{2(\log2){
(1+\epsilon)}m}\leq \mu(E_{m})\leq \big(1-\mu(B_m)\big)^{(\log2)(1-\epsilon)m}
$$
hold for all but finitely many $m\in\N$.
}
\end{Corollary}

\begin{proof}
{
It follows from \eqref{eq:gauss} that
\[
\frac{1}{k_m}=e^{(\log 2)\mu(B_{m})}-1.
\]
Since $e^x-1=x+\Oh(x^2)$, we obtain
\[
\frac{1}{k_m}=(\log 2)\mu(B_{m})+\Oh(\mu(B_m)^2).
\]
This also gives
\[
\frac{1}{k_m+1}=(\log 2)\mu(B_{m})+\Oh(\mu(B_m)^2).
\]
Consequently, there exists a positive constant $C$ such that
\[
1-2(\log 2)\mu(B_m)- C\mu(B_m)^2\leq 1-\frac{2}{k_m} 
\]
and
\[
1-\frac{1}{k_m+1}\leq 1-(\log 2)\mu(B_m)+C\mu(B_m)^2
\]
{From 
\eqref{eqn_meas_of_E_nm_gauss_11}} we {then} obtain
\begin{equation}
\begin{aligned}
\label{eqn_bound_in_measure_01}
\left(1-2(\log2)\mu(B_m)-C\mu(B_m)^2\right)^n &\leq \mu(E_{n,m})\\&\leq \left(1-(\log2)\mu(B_m)+C\mu(B_m)^2\right)^n.
\end{aligned}
\end{equation}
The claim now follows by taking $m=n$ in \eqref{eqn_bound_in_measure_01} and applying the inequalities
$$(1-x)^{r(1+\epsilon)}\leq 1-rx-Cx^2\text{ and }(1-x)^{r(1-\epsilon)}\geq 1-rx+Cx^2,$$ 
which hold for all $r>0$ and all sufficiently small positive $x$.
}
\end{proof}

\begin{Remark}
\label{rem_enm_gauus}
We are also interested in 
{estimating} the measure of sets of the form $E_{m_{j-1},m_j}$. In particular, if
$$
m_j=\lfloor b^j \rfloor
$$
for some $b>1$, then 
{the right-hand inequality in \eqref{eqn_bound_in_measure_01} yields}
\begin{equation}
\label{eqn_emj-1mj_estimate_gauss}
\begin{split}
\mu(E_{m_{j-1},m_j}) &
{\leq}\left(1-{(\log2)}\mu(B_{m_j})+\Oh(\mu(B_{m_j})^2)\right)^{m_{j-1}}
\\ &\leq \left(1-\mu(B_{m_j})\right)^{{{m
_j(\log2)}/{c}}}
\end{split}
\end{equation}
as long as 
{$c>b$}. 
\end{Remark}

\begin{proof}[Proof of \cref{thm_gauss}]
We begin with the case where there exists $\sigma<1$ such that $k_m\leq \frac{\sigma m}{\log\log m }$ for all but finitely many $m\in\N$.
Let $\sigma'$ be any number satisfying {$\sigma<\sigma'<1$}, and define {$C\coloneqq \frac{\sigma'}{\sigma\log 2}$}. Since {$\mu(B_m)
\geq \frac{\sigma'}{k_m\log2}$} for all but finitely many $m$, it follows that
$$
\mu(B_m)\geq \frac{\sigma'\log\log m}{\sigma m\log2 }=\frac{C\log\log m}{m}.
$$
Then, repeating an analogous argument to the one used in the proof of \cref{cor_targetindep_1}, we can show that
$$
\sum_{j\in\N} \left(1-\mu(B_{m_j})\right)^{m_j(\log2)/c}<\infty
$$
for any $b,c\in [1,C)$ with $b<c$, where $m_j=\lfloor b^j\rfloor$. In view of \eqref{eqn_emj-1mj_estimate_gauss}, this means that
$\sum_{j\in\N}\mu(E_{m_{j-1},m_j})<\infty$. Using \cref{lem_conull}, we conclude that $\eah(\targets)$ has full measure.

Next, we deal with the case when there exists $\tau>0$ such that $k_m\geq (\log m)^{\tau}$ for all but finitely many $m\in\N$.
Set $F(m):=\frac{1}{(\log m)^{1+\tau/2}\mu(B_m)}$, and note that $F(m)$ converges to $\infty$ as $m\to\infty$, because of the assumption that $k_m\geq  (\log m)^{\tau}$.
Moreover, $F$ satisfies \eqref{eqn_targetindependence_longterm_almost_mathcalH} by construction and, as explained at the beginning of this section, \eqref{eqn_targetindependence_longterm_almost} is satisfied for $\eta(m)=\Oh\left(\exp\big(-C{\sqrt{{F}(m)}}\,\big)\right)$. Here   it is important that $\lim_{m\to\infty}F(m)=\infty$, since this implies $\lim_{m\to\infty}\eta(m)=0$.
In light of \cref{cor_meas_of_E_nm_gauss} we have
$$
\sum_{m=1}^\infty \frac{\big(1-\mu(B_{m})\big)^{(\log2)(1-\epsilon)m}}{m}<\infty
~\implies~
\sum_{m=1}^\infty\frac{\mu(E_m)^{1-\frac{\epsilon}{2}}}{m}<\infty
$$
as well as
$$
\sum_{m=1}^\infty \frac{\big(1-\mu(B_{m})\big)^{
{2(\log2)(1+\epsilon)m}}}{m}
=\infty
~\implies~
\sum_{m=1}^\infty\frac{\mu(E_m)}{m}=\infty.
$$
Hence \cref{thm_gauss} follows directly from \cref{thm_targetindep_longterm_almost}.
\end{proof}

{We omit the proof of \cref{cor_targetindep_3}, since,  {with the help of \equ{gauss},} it can be derived from \cref{thm_gauss} in the same way that \cref{cor_targetindep_1} was derived from \cref{thm_targetindep_1}. }

\section{Further explorations and open questions}

There are still a multitude of intriguing questions surrounding the behavior of eventually always hitting sets.
We begin with the following.

\begin{Question}
\label{q_1}
Is it possible to upgrade {Theorem  
{\ref{thm_targetindep_longterm_almost}}
to include necessary and sufficient conditions for $\eah({\targets})$ to be a null/co-null set, perhaps with some mixing condition different from \eqref{eqn_targetindependence_longterm_almost}--\eqref{eqn_targetindependence_longterm_almost_mathcalH}?}
\end{Question}

\smallskip

Another intriguing question concerns rotations on the torus. Fix $\alpha\in [0,1)$ and consider the shrinking target system where $X$ equals the torus $\T$, the transformation is given by $T(x)=x+\alpha\bmod 1$, $\mu$ is Lebesgue measure, and 
\eq{torustargets}{\targets=(B_n)\text{ with }B_n\coloneqq \{x\in \T: \|x\|_\T < \psi (n)\},}
where $\psi\colon \N\to [0,1]$ is some non-increasing function.
In this case, the set of eventually always hitting points can be written as
\begin{equation}
\label{eqn_ea_kurzweil}
\eah(\targets) =\big\{y\in \T: \min_{1\leq k\leq n}\|k\alpha-y\|_\T<\psi(n)~\text{eventually always} \big\}.
\end{equation}

In \cite{KL} the Hausdorff dimension of $\eah(\targets)$ was computed for the cases where $\psi(n)=n^{-\tau}$ for some $\tau>0$. Closely related to the study of \eqref{eqn_ea_kurzweil} are also questions regarding inhomogeneous versions of Dirichlet's classical approximation theorem addressed in \cite{KW17, KK}.

As was mentioned in \cref{sec_intro}, Kurzweil \cite{Ku} proved that when $\alpha$ is badly approximable the hitting set $\h(\targets)$ for $\targets$ as in \equ{torustargets}
obeys the zero--one law
\begin{equation*}
\sum_n \psi(n)\begin{cases} < \infty\\ = \infty\end{cases} \iff \quad\h({\targets})\text{ has }\begin{cases} \text{zero}\\ \text{full} \end{cases}\text{measure.}
\end{equation*}
More recently, an extension of Kurzweil's result to arbitrary $\alpha\in [0,1)$ was given by Fuchs and Kim \cite{FK}:
\begin{equation*}
\sum_{k=1}^\infty\left(\sum_{n=q_k}^{q_{k+1}-1} \min \{ \psi(n), \|q_k \alpha\|_\T\}\right)\begin{cases} < \infty\\ = \infty\end{cases} \iff \quad\h({\targets})\text{ has }\begin{cases} \text{zero}\\ \text{full} \end{cases}\text{measure,}
\end{equation*}
where $p_k/q_k$ denote the principal convergents of $\alpha$.

By \cref{cor_01law}, we know that $\eah(\targets)$ also obeys a zero--one law. This leads to the following question.  

\begin{Question}
\label{q_3}
For a fixed $\alpha$ (at least in the case when $\alpha$ is badly approximable) what are necessary and sufficient conditions on $\psi$ so that the set $\eah(\targets)$ as in \eqref{eqn_ea_kurzweil} is a null set (or co-null set respectively)?
\end{Question}




Another classical type of shrinking target systems 
are $\beta$-transformations. Let $X=\T$ and, for $\beta >1$, consider the map $T_\beta(x)=\beta x\bmod 1$ alongside the 
shrinking targets on the torus $\T$ given by \equ{torustargets}.
In this set-up,
\begin{equation}
\label{eqn_ea_beta_trans}
\eah(\targets) =\big\{y\in \T: \min_{1\leq k\leq n}\|T_\beta^k(y)\|_\T<\psi(n)~\text{eventually always} \big\}.
\end{equation}

The Hausdorff dimension of the set $\eah(\targets)$ in \eqref{eqn_ea_beta_trans} was studied in \cite{BL}.
Unlike rotation by $\alpha$, the map $T_\beta$ is highly mixing, which suggests the following question.

\begin{Question}
\label{q_4}
Does $T_\beta$ and $\targets$ as above satisfy condition \eqref{eqn_targetindependence_longterm_almost}, perhaps with some additional assumptions on $\psi$?
\end{Question}

An affirmative answer to \cref{q_4} could lead to a better understanding of necessary and sufficient conditions for $\eah(\targets)$ in \eqref{eqn_ea_beta_trans} to have full or zero measure respectively, similar in spirit to Theorems \ref{thm_bernoulli} and \ref{thm_gauss}.



\bigskip
\footnotesize

\noindent
Dmitry\ Kleinbock\\
\textsc{Brandeis University}\par\nopagebreak
\noindent
\href{mailto:kleinboc@brandeis.edu}
{\texttt{kleinboc@brandeis.edu}}
\\

\noindent
Ioannis\ Konstantoulas\\
\textsc{Uppsala University}\par\nopagebreak
\noindent
\href{mailto:ykonstant@gmail.com}
{\texttt{ykonstant@gmail.com}}
\\

\noindent
Florian Karl\ Richter\\
\textsc{{\'E}cole Polytechnique F{\'e}d{\'e}ral de Lausanne (EPFL)}\par\nopagebreak
\noindent
\href{mailto:f.richter@epfl.ch}
{\texttt{f.richter@epfl.ch}}

\end{document}